\documentclass[12pt]{article}

\usepackage{amsfonts,amsmath, amssymb}

\usepackage{xcolor}

\usepackage{enumerate}
\usepackage{epsf,epsfig,amsfonts,a4wide}
\usepackage{amsfonts, mathdots}
\usepackage{amsmath,amssymb,amsthm}
\usepackage{url}
\usepackage{amsmath,amssymb,amsthm}

\usepackage{amssymb}
\usepackage{amsthm}
\usepackage{epsf,epsfig,amsfonts,graphicx}

\newtheorem{Theorem}{Theorem}[section]
\newtheorem{Lemma}{Lemma}[section]
\newtheorem{Proposition}{Proposition}[section]
\newtheorem{Corollary}{Corollary}[section]

\newtheorem{Ex}{Example}[section]
\newtheorem{Remark}{Remark}[section]

\theoremstyle{remark}

\newcommand{\be}{\begin{equation}}
\newcommand{\ee}{\end{equation}}

\newcommand{\R}{\mathbb{R}}\newcommand{\Id}{\textrm{\rm Id}}

\newcommand{\gl}{\mathrm{gl}}
\newcommand {\Lfirst} {{L_{\mathsf{comp1}}}}
\newcommand {\Lsecond} {{L_{\mathsf{comp2}}}}

\newcommand{\ddd}{\mathrm{d}}

\newcommand{\pd}[2]{\frac{\partial#1}{\partial#2}}

\newcommand{\dd}{{\mathrm d}\,}

\newcommand{\tr}{\operatorname{tr}}

\newcommand{\weg}[1]{}

\title{{Nijnehuis Geometry III: $\gl$-regular Nijenhuis operators}}
\author{Alexey V. Bolsinov\footnote{ School of Mathematics,
 Loughborough University,
 LE11 3TU, UK;  Faculty of Mechanics and Mathematics, Moscow State University and Moscow Center
for Fundamental and Applied Mathematics,  119992,  Moscow  Russia\ \
 \quad {\tt A. Bolsinov@lboro.ac.uk} } \quad
\& \quad  Andrey Yu. Konyaev\footnote{Faculty of Mechanics and Mathematics, Moscow State University, 119992, Moscow Russia
 \ \ \quad {\tt  maodzund@yandex.ru}} \quad \& \quad Vladimir S. Matveev\footnote{
Institut f\"ur Mathematik, Friedrich Schiller Universit\"at Jena,
07737 Jena Germany  \ \ \quad {\tt  vladimir.matveev@uni-jena.de}} 
}  
\date{}

\begin{document}

\maketitle

\begin{abstract}
We study Nijenhuis operators, that is, $(1,1)$-tensors with vanishing Nijenhuis torsion under the additional assumption that they are $gl$-regular, i.e., every eigenvalue has geometric multiplicity one. We prove the existence of a coordinate system in which the operator takes first or second companion  form, and give a local  describtion of such operators. We apply this local description to study singular points. In particular, we obtain their normal forms in dimension two and discover topological restrictions for the existence of $gl$-regular  Nijenhuis operators on closed surfaces. 

This paper is  an important step in the research programme suggested in \cite{Nijenhuis1,openprob}.

\end{abstract}

\tableofcontents


\section{Basic definitions and main results}\label{sect:intro2}

Given a $(1,1)$ tensor field $L$  on a manifold $\mathsf M^n$, one defines the {\it Nijenhuis torsion} of $L$ as
\begin{equation}
  \mathcal N_L(\xi,\eta) = L^2[\xi,\eta] - L[L\xi,\eta] - L[\xi,L\eta] + [L\xi, L\eta],  
\end{equation}
where $\xi, \eta$ are arbitrary vector fields.  Recall that $L$ is said to be a {\it Nijenhuis operator} if its Nijenhuis torsion vanishes. 


Nijenhuis geometry studies Nijenhuis operators and their properties, both local and global. A research programme and general strategy for studying such operators were suggested in \cite{Nijenhuis1}.  This paper is devoted to the next item of our agenda (after \cite{Nijenhuis1}, \cite{Nijenhuis3}, \cite{konyaev})  and is focused on Nijenhuis operators satisfying $\gl$-regularity condition.

We start with the following equivalent definitions of $\gl$-{\it regular} operators $L:\R^n \to \R^n$ (the same notation $L$ will be used for the matrix corresponding to this operator, with appropriate amendments under coordinate transformations if necessary):  
\begin{itemize}

\item $L$ is a regular element of the Lie algebra $\gl(n,\R)$ in the sense that the adjoint orbit $\mathcal  O(L)=\{ PLP^{-1} ~|~ P\in \mathrm{GL}(n,\R)\} \subset \gl(n,\R)$ has maximal dimension.

\item The operators $\operatorname{Id}, L, \dots, L^{n - 1}$ are linearly independent.

\item For each eigenvalue of $L$ there is exactly one Jordan block in its Jordan normal form (this includes complex eigenvalues).
    
    \item The minimal polynomial of $L$ coincides with the characteristic polynomial
    $$
    \chi_L (\lambda) =  \det (\lambda{\cdot}\operatorname{Id} - L) = \lambda^n - c_1 \lambda^{n - 1} - \dots - c_n.
    $$
    \item $L$ is similar to the \textit{first companion form}
    \begin{equation*}
    \left(\begin{array}{ccccc}
     c_1 & 1 & 0 & \dots & 0  \\
     c_2 & 0 & 1 & \ddots & \vdots  \\
     \vdots & \vdots  & \ddots & \ddots & 0 \\
     c_{n - 1} & 0 & \dots & 0 & 1 \\
     c_n & 0 & \dots  & 0  & 0 \\
    \end{array}\right),
    \end{equation*}
    where $c_i$ are the coefficients of the characteristic polynomial $\chi_L (\lambda)$.
    \item $L$ is similar to the \textit{second companion form}
        \begin{equation*}
        \left(\begin{array}{ccccc}
     0 &  \! 1 & \!\!\! 0 & \dots & 0  \\
     0 & \! 0 & \!\!\! 1 & \ddots & \vdots \\
     \vdots & \!  \vdots & \!\! \ddots \! &\ddots &  0\\
     0 &  \! 0 & \!\! \dots \! & 0 & 1 \\
     c_n &  \!\! c_{n - 1} \! & \!\! \dots \! & c_2 & c_1 \\
    \end{array}\right),
\end{equation*}
where $c_i$ are the coefficients of the characteristic polynomial $\chi_L (\lambda)$.
\end{itemize}

We say that a Nijenhuis operator $L$ defined on a smooth manifold $\mathsf M$ is \textit{$\gl$-regular}, if it is $\gl$-regular at every point $\mathsf p \in \mathsf M$ \cite[Definition 2.9]{Nijenhuis1}.  Many results in our paper are local and in this case $\mathsf M$ is an open domain in $\R^n$.

Note that the eigenvalues of $\gl$-regular operators  are not necessarily smooth as the following example shows. 
Consider the $\gl$-regular Nijenhuis operator
\begin{equation*}
    L = \left(\begin{array}{cc}
         x & 1  \\
         y & 0 
    \end{array}\right)
\end{equation*}
 on $\R^2(x,y)$. Its eigenvalues are
$$
\lambda_{1, 2} = \frac{x \pm \sqrt{x^2 + 4y}}{2}.
$$
On the curve $x^2 + 4y = 0$,  $L$ is similar to a single Jordan block with eigenvalue $\frac{x}{2}$. If $x^2 + 4y > 0$, then $L$ is semisimple with distinct real eigenvalues (thus, $\R$-diagonalizable) whereas for $x^2 + 4y < 0$ this operator has two complex conjugate eigenvalues.  In particular,  this shows that $\gl$-regular operators may admit singular points (cf. \cite[Definition 2.8]{Nijenhuis1}) at which the algebraic structure of $L$ changes.

All the objects we are dealing with are supposed to be real analytic.  The first result of the paper is the following theorem which gives a local characterisation of  $\gl$-regular Njenhuis operators of any algebraic type.

\begin{Theorem}\label{main:1}
Consider a real analytic $\gl$-regular operator $L$ with characteristic polynomial 
$$
\chi_L (\lambda) = \det (\lambda{\cdot}\operatorname{Id} - L) = \lambda^n - f_1 \lambda^{n - 1} - \dots - f_n
$$ 
for $n \geq 2$ in a sufficiently small neighbourhood of a point $\mathsf{p}\in \mathsf{M}$.  Then the following are equivalent

\begin{enumerate}
\item[{\rm (i)}] $L$ is Nijenhuis.
    \item[{\rm (ii)}] There exists a local coordinate system $x=(x^1, \dots, x^n)$ in which $L$ takes the following form
    \begin{equation}\label{first}
  \Lfirst (x) =  \left(\begin{array}{ccccc}
     f_1 & 1 & 0 & \dots & 0  \\
     f_2 & 0 & 1 & \ddots & \vdots \\
     \vdots & \vdots & \ddots & \ddots & 0 \\
     f_{n - 1} & 0 & \dots & 0 & 1 \\
     f_n & 0 & \dots &  0 & 0 \\
    \end{array}\right),
    \end{equation}
    where $f_i=f_i(x)$ are coefficients of the characteristic polynomial in this coordinate system. 
    These coefficients satisfy the following system of PDEs: 
    \begin{equation}\label{first_set}
    \begin{aligned}
        & \frac {\partial f_i} {\partial x^{j}} = f_i  \frac{\partial f_1}{\partial x^{j+1}} + \frac{\partial f_{i + 1}} {\partial x^{j+1}} , \\
        & \frac{\partial f_n}{\partial x^{j}}  = f_n \frac{\partial f_1}{\partial x^{j+1}}.
        \end{aligned}
    \end{equation}
for $1 \leq i, j \leq n-1$. 
     
       \item[{\rm (iii)}] There exists a local coordinate system $x=(x^1, \dots, x^n)$ in which $L$ takes the following form
    \begin{equation}\label{second}
    \Lsecond(x) = \left(\begin{array}{ccccc}
     0 & 1 & 0 & \dots & 0  \\
     0 & 0 & 1 & \ddots & \vdots \\
     \vdots & \vdots & \ddots &\ddots &  0\\
     0 & 0 & \dots & 0 & 1 \\
     f_n & f_{n - 1} & \dots  & f_2 & f_1 \\
    \end{array}\right),
    \end{equation}
    where $f_i=f_i(x)$ are coefficients of the characteristic polynomial in this coordinate system.  These coefficients satisfy a system of PDEs that can be written in the form
    \begin{equation}\label{second_set}
        \ddd \omega = 0,  \quad \ddd \bigl( L^*\omega \bigr)= 0,
    \end{equation}
where $\omega = f_n \ddd x^1 + \dots + f_1 \ddd x^n$.
    
   \end{enumerate}
\end{Theorem}

Following the terminology from Linear Algebra, we will refer to \eqref{first} and \eqref{second} as the {\it first} and {\it second companion forms} of $L$.

\begin{Remark}\label{rem:1.4}
{\rm
If a Nijenhuis operator $L$ is differentially non-degenerate at a point $\mathsf{p}\in \mathsf{M}$ \cite[Definition 2.10]{Nijenhuis1}\footnote{Recall that this condition means that the differentials $\ddd f_1(\mathsf{p}), \dots ,\ddd f_1(\mathsf{p})$ are linearly independent.}, then there are two distinguished coordinate systems in which $L$ takes the first and second companion form.  Namely,  if we take the coefficients of the characteristic polynomial of $L$ as local coordinates, i.e.,  set $x^i = f_i$, then in these coordinates $L$ takes the form   
\begin{equation}
\label{eq:nilp_pert}
 \Lfirst (x) =  \left(\begin{array}{ccccc}
     x^1 & 1 & 0 & \dots & 0  \\
     x^2 & 0 & 1 & \ddots & \vdots \\
     \vdots & \vdots & \ddots & \ddots & 0 \\
     x^{n - 1} & 0 & \dots & 0 & 1 \\
     x^n & 0 & \dots &  0 & 0 \\
    \end{array}\right).
\end{equation}
Similarly, if we set $x^1 =\tr L, x^2 = \frac{1}{2}\tr L^2, \dots, x^n=\frac{1}{n}\tr L^n$, then in these coordinates, we have
$$
    \Lsecond(x) = \left(\begin{array}{ccccc}
     0 & 1 & 0 & \dots & 0  \\
     0 & 0 & 1 & \ddots & \vdots \\
     \vdots & \vdots & \ddots &\ddots &  0\\
     0 & 0 & \dots & 0 & 1 \\
     f_n(x) & f_{n - 1}(x) & \dots  & f_2(x) & f_1(x) \\
    \end{array}\right),
$$
where $f_i (x)$ are the so-called Newton-Girard polynomials that express the coefficients of the characteristic polynomial in terms of the traces of powers of $L$ appropriately rescaled, see \cite[Appendix B]{Nijenhuis3} for details.

The point of Theorem \ref{main:1}, however, is that such a nice companion form exists for any $\gl$-regular Nijenhuis operator so that in the real analytic category the differential non-degeneracy condition is not actually important.
}\end{Remark}

\begin{Remark}  {\rm The existence of the first companion form for an operator $L$ is equivalent to the existence of a vector field $\xi$ such that  $\xi$, $L\xi$, $L^2\xi$, \dots, $L^{n-1}\xi$ pairwise commute and are linearly independent  (for $\Lfirst$, this vector field is $\xi=\partial_{x^n}$).   Similarly, the existence of the second companion form for $L$ is equivalent to the existence of a closed 1-form $\alpha$ such that the forms $\alpha$, $L^*\alpha$, $(L^*)^2\alpha$,\dots, $(L^*)^{n-1} \alpha$ are all closed and linearly independent (for $\Lsecond$,  we can take $\alpha = \ddd x^1$).  
}\end{Remark}

\begin{Remark}{\rm The reducibility of an  operator to a companion form
by a coordinate transformation is a
 non-trivial condition. Indeed,  
companion forms \eqref{first} and \eqref{second} are parametrised by $n$ functions (in $n$ variables). The coordinate change is also parametrized by $n$ functions. At the same time, an operator field $L$ (not necessarily Nijenhuis) is parametrised by $n^2$ functions. For $n > 2$ one has $n^2 > 2n$ and, thus, almost no operator field $L$ can be brought to companion form.

As a specific example, consider $L$ such that the coefficients $f_i$ of its characteristic polynomial $\chi_L(\lambda)$ are all constant. The companion form for $L$ will then be a constant matrix.  Hence, if $L$ is reducible to companion form by a suitable coordinate transformation then its Nijenhuis torsion $\mathcal N_L$ necessarily vanishes, which is not always the case. Indeed, take
$$
    L = \left(
    \begin{array}{ccc}
         0 & 1 & 0  \\
         -(y^2 + 1) & 0 & 1\\
         0 & (y^2 + 1) & 0 \\
    \end{array}
    \right).
$$
This operator is nilpotent, but $\mathcal N_L \neq 0$. Thus, $L$ cannot be brought to companion form.
}\end{Remark}

\begin{Remark}{\rm
The set of coordinate systems in which $\gl$-regular Nijenhuis operator $L$ is in first or second companion form is parametrised by $n$ functions of one variable which is the maximal number of possible parameters. In more precise terms, the equations defining the corresponding coordinate transformations (see \eqref{eq:bols3.2} and \eqref{eq:bols3.4} below) are in involution for a $\gl$-regular operator $L$ if and only if $L$ is Nijenhuis (see Propositions \ref{prop:3.2} and \ref{pro1}).  
}\end{Remark}


Theorem \ref{main:1}  characterises $\gl$-regular Nijenhuis operators but, in fact, should not be interpreted as their local description. To get  such a description one needs another important step. Namely, one needs to resolve PDE system  \eqref{first_set}  in order fo find functions $f_i$ from the first column of $\Lfirst$. The second result of our paper is an algebraic method for solving this system for arbitrary initial conditions.  

\begin{Theorem}
 \label{main:2}
 For $n$ arbitrary real analytic functions $v_1(t), \dots, v_n(t)$ defined in a neighbourhood of zero,   
 consider the function 
 $$
 r(\lambda,t) = \lambda^n  -  v_1(t) \lambda^{n-1} - v_2(t) \lambda^{n-2} - \dots - v_{n-1}(t) \lambda - v_n(t)
 $$ 
 and the matrix relation 
 $$
 r(L, M) = 0,
 $$
 where $M = x^1 L^{n-1} + x^2 L^{n-2} + \dots + x^{n-1} L + x^n\operatorname{Id}$ and $L$ is a $\gl$-regular  $n\times n$ matrix. 
 Then
 \begin{itemize}
\item From this matrix relation,  the coefficients $f_1,\dots, f_n$ of the characteristic polynomial of $L$ can be uniquely expressed in a neighbourhood of $x=0$  as real analytic functions in $x^1,\dots, x^n$ (by Implicit Function Theorem).   

\item The functions $f_1(x),\dots,f_n(x)$ so obtained are solutions  of \eqref{first_set} satisfying the initial condition 
 \begin{equation}
 \label{eq:27}
 \begin{aligned}
 f_1(0,\dots,0,x^n) &= v_1(x^n),\\
 f_2(0,\dots,0,x^n) &= v_2(x^n),\\
 & \dots \\
 f_n(0,\dots,0,x^n) &= v_n(x^n).
 \end{aligned}
 \end{equation}
 \end{itemize}
\end{Theorem}

This theorem gives  local description for all $\gl$-regular Nijenhuis operators and therefore  provides a ``list'' of  all possible singularities that can occur for $\gl$-regular operators (Example \ref{exam:5.1} demonstrates how it works in practice).   One should, however,  remember that the first companion form for a Nijenhuis operator $L$ is not unique.  In other words,   different companion forms can be equivalent. Speaking in rigorous terms, on the space of all (Nijenhuis) companion forms $\Lfirst$ given by \eqref{first}, we can introduce a natural action of the groupoid that consists of coordinate transformations sending one companion form into another. Local classification of $\gl$-regular operators in proper sense amounts to the orbit classification for this action.  For $n\ge 3$, we hope to address this problem elsewhere.  

In the two-dimensional case, which is somehow rather special,  the local classification of $\gl$-regular Nijenhuis operators is obtained in Section \ref{sect:dim2}, Theorem \ref{thm:dim2}.  In addition to three (algebraically) generic types of $\gl$-regular operators, this theorem describes five types\footnote{The other two series  $L_{\mathrm{nil}}$ and $N$ from Theorem  \ref{thm:dim2}  are not singular as the algebraic type of these operators does not change, at each point the operator is a $2\times 2$ Jordan block.} of singular points  (series $L_{\mathrm{nc}}$, $M$, $O$, $P$ and $S$) for $\gl$-regular operators in dimension 2.   It appears that locally every  $\gl$-regular Nijenhuis operator can be reduced to an explicit polynomial canonical form,  which is quite different from the companion form.  Our choice is explained by the following natural reason.  The functions $f_1$ and $f_2$ involved in  $\Lfirst$  are solutions of \eqref{first_set}  and Theorem \ref{main:2} suggests that they can be found explicitly only in exceptional cases. Despite its elegance and convenience for various theoretical purposes, the companion form $\Lfirst$ does not provide description in elementary functions. However, such a description can be achieved by an appropriate change of variables and that is what  Theorem \ref{thm:dim2} does. 

Based on this theorem we obtain the following global description of Nijenhuis operators on closed two-dimensional manifolds.

\begin{Theorem}
\label{main:3}
Let  $(\mathsf{M}^2, L)$ be a closed connected $\gl$-regular Nijenhuis 2-manifold.   Then one of the following holds:
\begin{enumerate}
\item  $\mathsf{M}^2$ is orientable and  $L = \alpha \operatorname{Id} + \beta A$,  where $A$ is a complex structure on $\mathsf{M}^2$ and $\alpha, \beta\in \R$ are constants, $\beta\ne 0$.
\item  $\mathsf{M}^2$ is homeomorphic to either a torus or a Klein bottle and $L$ has two distinct real eigenvalues on $\mathsf{M}^2$ at each point.   
\item  $\mathsf{M}^2$ is homeomorphic to a torus and $L$ is similar to a Jordan block at each point of $\mathsf{M}^2$.
\item $\mathsf{M}^2$ is homeomorphic to either a torus or a Klein bottle and one of the eigenvalues of $L$ is constant.
\end{enumerate}
\end{Theorem}

In the first three cases, the algebraic type of $L$ remains the same at each point of the surface. In other words, the set of singular points is empty.  In the forth case, the eigenvalues of $L$ may collide and we show in Proposition \ref{prop4.2} that the corresponding singular point necessarily belongs to the $M$-series, one of five series from Theorem \ref{thm:dim2}. In particular,  the other types of singular points cannot occur on compact surfaces.  

Theorem \ref{main:3}  provides topological  obstructions for existence of (non-trivial) $\gl$-regular Nijenhuis operators in dimension 2.

\begin{Corollary}  Let $\mathsf{M}^2$ be either a sphere or a closed Riemann surface of genus $\ge 2$. Then $\mathsf{M}^2$ cannot carry any $\gl$-regular Nijenhuis operator $L$ except for $L = \alpha \operatorname{Id} + \beta A$,  where $A$ is a complex structure on $\mathsf{M}^2$ and $\alpha, \beta\in \R$, $\beta\ne 0$.
\end{Corollary}

\begin{Corollary} A non-orientable closed 2-manifold different from a Klein bottle cannot carry any $\gl$-regular Nijenhuis operator.
\end{Corollary}

Another result of our paper is description of various scenarios for Nijenhuis perturbations of a Jordan block. Assume that at a given point $\mathsf{p}$,  all the coefficients $f_1,\dots, f_n$ of the characteristic polynomial of  a Nijenhuis operator $L$ vanish so that $L(\mathsf{p})$ is similar to a Jordan block with zero eigenvalues.  What can we say about the algebraic type of $L$ at a generic point $\mathsf{q}\in U(\mathsf{p})$?   Formula \eqref{eq:nilp_pert}  gives an example when $L(\mathsf{q})$ typically becomes semisimple, moreover for any prescribed collection of eigenvalues $\lambda_1,\dots,\lambda_n$ (with arbitrary multiplicities and including complex conjugate pairs) there exists exactly one point $\mathsf{q}$ that realises this spectrum of $L$.
This scenario coincides with the versal deformation of a Jordan block  in terms of V.Arnold \cite{Arnold}.  But can $L$ split into two Jordan blocks?  Or, more generally, does there exist a Nijenhuis perturbation of a Jordan block $J_0=L(\mathsf{p})$ such that at a generic point  $\mathsf{q}\in U(\mathsf{p})$ the operator $L(\mathsf{q})$ has a prescribed algebraic type?

We  use Theorem \ref{main:2} to show that the answer is positive:  all scenarios are possible. To state this result in a rigorous way, recall that in the space of all $n\times n$ matrices,  which we interpret as the Lie algebra $\gl(n,\R)$,  we can introduce a natural partition $\gl(n,\R) = \sqcup_\alpha W_\alpha$ into families of adjoint orbits having the same algebraic type (Segre characteristic). Such families are sometimes called {\it layers}. For regular orbits, their algebraic type is defined by multiplicities $k_1,\dots, k_s$ of eigenvalues\footnote{Though we deal with real matrices, we make no difference between complex and real roots.}, so that we can write
$$
\gl(n,\R)^{\mathrm{reg}} = \bigsqcup_{\sum k_s = n} W_{k_1,\dots,k_s}, \qquad    k_1\le \dots \le k_s, \ s\in \mathbb N, \ k_i\in\mathbb N,
$$
where $W_{k_1,\dots,k_s}\subset \gl(n,\R)$ is the subset of $\gl$-regular operators having $s$ distinct eigenvalues with multiplicities $k_1,\dots, k_s$ (regularity will automatically imply that each eigenvalue contributes exactly one Jordan block into the Jordan normal form of the operator).  Notice that the Jordan block $J_0$ belongs to the closure of each regular layer. 

\begin{Theorem}
\label{main:4}
For any regular layer $W_{k_1,\dots,k_s}\subset \gl(n,\R)$ there exists a Nijenhuis operator $L$ defined in a small neighbourhood of $0\in\R^n$ such that $L(0)=J_0$ and $L(x)\in \overline{W}_{k_1,\dots,k_s}$ for all $x\in U(0)$, where $\overline{W}_{k_1,\dots,k_s}$ is the closure of $W_{k_1,\dots,k_s}$ (in usual or Zariski topology).
\end{Theorem}

The structure of the paper is as follows.  The proofs of Theorems \ref{main:1} and \ref{main:2} are given in Sections \ref{sect:proof1} and \ref{sect:proof2} respectively.  Section \ref{sect:perturb}  is devoted to Nijenhuis perturbations of a Jordan block and contain the proof of 
 Theorem \ref{main:4}. In Section \ref{sect:dim2} we obtain local classification of all $\gl$-regular Nijenhuis operators in dimension 2 and prove Theorem \ref{main:3}. These sections are mainly independent on each other  and contain no cross references.  We conclude the paper with Appendix devoted to some applications of Theorem \ref{main:2} to quasilinear systems of hydrodynamic type $u_t = L(u) u_x$ in the case when $L(u)$ is not necessarily diagonalisable Nijenhuis operator.

{\bf Acknowledgements.}  We thank Jenya Ferapontov and Artie Prendergast-Smith for their valuable comments and explanations.  
The most essential steps resulted in this paper would not have been done without outstanding research environment offered to us by the Institute of Advanced Studies, Loughborough University and Centro Internazionale per la Ricerca Matematica, Trento.  
We are also grateful to Jena Universit\"at, in particular, Ostpartnerschaft programm for supporting our research on Nijenhuis Geometry for several years. The work of Alexey Bolsinov and Andrey Konyaev was  supported by Russian Science Foundation (project 17-11-01303).


\section{Outlook and motivation}\label{sect:intro}

Our motivation for studying $\gl$-regular Nijenhuis operators was based on a very naive question: ``What is the most natural genericity assumption for  $(1,1)$-tensor fields similar to non-degeneracy  of bilinear forms, symmetric or skew-symmetric?''  In general algebraic context,  the latter condition simply means that a bilinear  form belongs to the ``largest'' orbit of the natural $\mathrm{GL}(n)$-action and hence is the most typical. As a matter of fact such an orbit, in this case, is open.  For operators, there are no open orbits,  but we may still consider $\mathrm{GL}(n)$-orbits of maximal dimension, which is exactly the $\gl$-regularity assumption\footnote{The non-degeneracy assumption $\det L\ne 0$ is much less relevant in Nijenhuis geometry as many problems one has to deal with are invariant w.r.t. shifts 
$L \mapsto L + \mathrm{const}\cdot\Id$.
}.  
In this view,  $\gl$-regular operators can be thought of as natural analogs of symplectic forms and (pseudo)-Riemannian metrics.  

Another naive way to look at $(1,1)$-tensor fields is to think of them as families of matrices depending on parameters (coordinates on the manifold). Then the next natural question would obviously be:  ``Which bifurcations are typical in such families?''.  The most typical bifurcation is a collision of two (or several) eigenvalues resulting in appearance of a Jordan block.  That is exactly a singularity which we may observe in the case of $\gl$-regular Nijenhuis operators.  One could, of course, avoid collision of eigenvalues by requiring that $L$ has no multiple eigenvalues, but would make the definition too rigid and exclude many important examples and interesting phenomena. It is worth mentioning that the complement to the set of matrices with no multiple roots has codimension one, whereas the complement to the set of $\gl$-regular matrices  is much smaller and has codimension 3. 

The ``converse'' question, naturally appearing in applications, can be stated as follows: ``What happens to a Jordan block under a perturbation?''.  The answer depends on the number of parameters involved in perturbation and additional assumptions imposed on it.  We refer to the famous paper by Arnold \cite{Arnold} devoted to this subject which contains, in particular, an elegant solution in terms of versal deformations. In the context of Nijenhuis geometry,  it is quite natural to ask ``What are {\it Nijenhuis} perturbations of a Jordan block? Can we describe {\it all} of them?  Which of them are generic ({\it versal} in the sense of Arnold)?''   This is again a question on $\gl$-regular Nijenhuis operators.  It is amazing that the answer turns out to be very similar to that given by Arnold:   there is a very simple generic Nijenhuis perturbation of a Jordan block (see formula \eqref{eq:nilp_pert} and Proposition \ref{prop:5.1}), which is unique and coincides exactly with the one given in  \cite{Arnold}.  All the others can be derived from this canonical one by solving a system of integrable PDEs.  We give a purely algebraic algorithm (see Theorem \ref{main:2}) how to do it for arbitrary initial condition, i.e., for finding {\it all} the solutions.  

We also want to emphasise that $\gl$-regular operators share many common properties.  There are many facts well known for diagonalisable operators with simple spectrum that still hold true for $\gl$-regular operators.   If an operator is diagonalisable almost everywhere and has no multiple eigenvalues, then some  (but not all!) of these results can be transferred to $\gl$-regular case by continuity.  However, even this procedure is often non-trivial as one needs to show that ``transferring objects'',  e.g.  conservation laws or commuting flows,  remain smooth and independent (linearly or functionally or otherwise), i.e., they neither explode nor blow up.   Moreover, there are many occasions when a given Nijenhuis operator is not diagonalisable at all,  but $\gl$-regularity still guaranties good properties. 

For this reason we are trying to use ``invariant language'' in our proof. This makes things technically a bit more complicated (for Nijenhuis operators written in diagonal form some of our proofs would be just one line) but, as a reward, we manage to cover many different cases by using one universal approach suitable for all Nijenhuis operators satisfying just one additional condition, namely $\gl$-regularity.


We are confident that our results can and will have many applications. Indeed, Nijenhuis operators naturally appear in many unrelated topics in differential geometry and  mathematical physics.  A possible explanation for this ``experimentally observed phenomenon'' is as follows. For  many  geometric  systems  of partial differential  equations,  their  coefficients are constructed from a certain operator, i.e., $(1,1)$--tensor field $L=\bigl(L^i_j(u)\bigr)$.   If such a system is invariant with respect to diffeomorphisms, then the compatibility and involutivity conditions can be invariantly written   in terms of $L$.  The point is that vanishing of the Nijenhuis torsion of $L$ is, in a certain sense (see e.g. discussion in the introduction of \cite{Nijenhuis1}), the simplest non-trivial condition of this kind. 

This ``experimental observation'' suggests that any progress in Nijenhuis geometry might and should be applied in different areas where Nijenhuis operators have appeared, by combining the questions/methods from those topics with new results on Nijenhuis operators.

Until very recently,  the list of known 
 results in Nijenhuis geometry was very limited: Haantjes  theorem \cite{haant}, Newlander--Nirenberg theorem \cite{nirfro} and Thompson theorem  \cite{Thompson}.  
These results have been extensively used as a simplifying ansatz in those situations where Nijenhuis operators appear: customary, one works with those coordinates in which  
the operator takes the ``best''  possible form provided by these theorems (e.g., in the case of Haantjes theorem, 
 $L$ reduces to diagonal form with diagonal elements  $\lambda_i=\lambda_i(u_i)$ and  in the case of Thomson and  Newlander--Nirenberg Theorems, one works in a coordinate system where $L$ has constant entries).

The assumptions  of   Haantjes, Newlander--Nirenberg  and   Thompson theorems  essentially limit  their applications.  They all  require that $L$ is algebraically stable, i.e.,   has the same Segre characteristic at every point. Moreover, they have strong conditions on the Segre characteristic: in  Haantjes  and Newlander--Nirenberg  theorems,  the operator $L$ is semi-simple (diagonalisable  over complex numbers).  Thompson and Newlander--Nirenberg theorems  assume that the eigenvalues of $L$ are constant.

This paper, as well as its  predecessors \cite{Nijenhuis1, Nijenhuis3, konyaev}, 
aim to repair this  situation. An important ingredient  of our  strategy described in \cite{Nijenhuis1} is to develop tools to  study and describe Nijenhuis operators near  those points where the Segre characteristic changes  (singular points in the 
terminology of \cite{Nijenhuis1}) and also on closed manifolds.  Any such tool can be applied wherever Nijenhuis operators naturally appear.

Obviously, there are many different types of singularities for Nijenhuis operators. We started our research with two opposite cases: the paper  \cite{konyaev} (see also \cite[\S 5]{Nijenhuis1})  studies the so-called singular points of scalar type, i.e., those where
 the operator $L$ vanishes (we may think of them as the {\it most} singular points). In the present 
 paper  we come from the other side and consider singular points at which the operator $L$ remains $\gl$-regular  \cite[Definition 2.9]{Nijenhuis1} (the {\it least} singular points). 
Our first main result, Theorem \ref{main:1}, provides a common framework for studying such singularities:   it allows one to assume without loss of generality that $L$ locally takes the first or second 
companion form (see \eqref{first} and \eqref{second}). Similar to the diagonal form from the Haantjes theorem, the companion forms  \eqref{first} and \eqref{second} depend on an arbitrary choice of $n$ functions of one variable. In contrast to the Haantjes theorem, they allow bifurcations of the eigenvalues, and in Section \ref{sect:perturb}  we discuss   the freedom in such bifurcations.

 A demonstration that our strategy works is Theorem \ref{thm:dim2}  that describes all possible singularities for $\gl$-regular Nijenhuis operators in dimension 2. As a corollary we have  Theorem \ref{main:3} on topological obstructions for the existence of regular Nijenhuis operators on closed two-dimensional surfaces.

We expect many applications of our results. For example, with the help of  Theorem \ref{main:3} one can easily reprove most results of the paper \cite{Matveevjpn} devoted to geodesically equivalent metrics on  two dimensional semi-riemannian manifolds.  By \cite{Benenti},  a pair of such metrics allows one to construct a Nijenhuis operator. One can easily show, applying the trick from \cite[\S 3.3]{Matveev2016}, that  on a closed surface this operator is always $\gl$-regular provided the metrics are semi-riemannian.  Case 1   of  Theorem \ref{main:3} corresponds to a trivial geodesic equivalence, and  cases 2, 3 and 4, translated to the language of geodesically equivalent metrics, imply most results  in \cite{Matveevjpn} and   in particular allow  to prove the natural generalisation of the projective Obata conjecture for the 2-torus.

We expect that our results may be effectively used in the theory of (infinite-dimensional) integrable systems of hydrodynamic type. They are  partial differential equation systems of the form 
\begin{equation}
\label{eq:hydro} 
u^i_t= \sum_{j} A_{j}^{i}(u)  u^j_{x}.
\end{equation}
where $u(t,x)= (u^1(t,x),...,u^n(t,x))$ is an unknown vector-function. In this case the matrix $A=A(u)$ can be seen as an operator on an $n$-dimensional manifold with local coordinates $(u^1,...,u^n)$.  The integrability of this system amounts to a certain condition on the operator $A$ (more general than vanishing of the Nijenhuis torsion, see \cite{Tsarev}).

 One of the standard  methods to work with systems \eqref{eq:hydro} is based on the so-called Riemann invariants which are closely related to finding a polynomial $p$  with coefficients depending on $u$ such that $p(A)$ is a  Nijenhuis operator (the eigenvalues of the operator $p(A)$ are precisely the Riemann invariants).

 The overwhelming majority of results on integrable systems of hydrodynamic type assume that the operator $A$ is simple (i.e., has $n$ different eigenvalues).  Our results allow one to avoid this assumption.  In particular, they can be applied to study stability of solutions 
 of  \eqref{eq:hydro} near the points where the eigenvalues collide. The ``proof of concept'' is, in fact, the Appendix where we demonstrate how it works in the simplest case, when the operator $A$ is itself a Nijenhuis operator. 

Notice that not diagonalisable but still $\gl$-regular operators naturally appear in differential geometry and mathematical physics in the context of integrable PDEs of type \eqref{eq:hydro}, see e.g. \cite{Fordy, Benney2, BialyMironov,  Benney, Tsarev}.   Moreover, they often resemble the companion form discussed in Theorem \ref{main:1}.


\section{Proof of Theorem \ref{main:1}}\label{sect:proof1}

First of all we observe that every operator  $\Lfirst$ given by \eqref{first} is Nijenhuis if and only if  relations \eqref{first_set} hold. And similarly, every operator  $\Lsecond$ given by \eqref{second} is Nijenhuis if  and only if \eqref{second_set} holds.  The verification of this fact is straightforward and we omit it.  In terms of Theorem  \ref{main:1} this means, in particular, that (ii) $\Rightarrow$ (i) and (iii) $\Rightarrow$ (i).    

It remain to show that every $\gl$-regular Nijenhuis operator $L$ can be (locally) reduced to either of the companion forms $\Lfirst$ and $\Lsecond$.   Since the proofs for $\Lfirst$ and $\Lsecond$ are rather similar, we will do reduction simultaneously for both of them following the same scheme.
 
Consider a $\gl$-regular Nijenhuis operator $L$ in a neghbourhood $U(\mathsf{p})$ of a point $\mathsf{p}\in \mathsf{M}$  and choose local coordinates $u=(u^1,\dots, u^n)$ in this neighbourhood.  Our goal is is to find coordinate transformations bringing $L$ to the first companion form \eqref{first} and second companion form \eqref{second}.

For the first companion form, such a coordinate transformation $u=u(x)$,  where $x=(x^1,\dots, x^n)$ is a new coordinate system, satisfies the following system of PDEs:
\begin{equation}
\label{eq:bols3.1}
{\left( \frac{\partial u}{\partial x}\right)}^{-1} L(u) \, {\left( \frac{\partial u}{\partial x}\right)} = {\Lfirst} (x),
\end{equation}
where ${\Lfirst}$ stands for the first companion form \eqref{first} and $\left( \frac{\partial u}{\partial x}\right)$ denotes the Jacobi matrix of the transformation $u=u(x)$:
\begin{equation*}
{\left( \frac{\partial u}{\partial x}\right)}   = \left(\begin{array}{cccc}
         u^1_{x^1} & u^1_{x^2} & \dots & u^1_{x^n}  \\
         u^2_{x^1} & u^2_{x^2} & \dots & u^2_{x^n}  \\
         \vdots & \vdots & \ddots & \vdots  \\
         u^n_{x^1} & u^n_{x^2} & \dots & u^n_{x^n}  \\
    \end{array}
    \right).
\end{equation*}
Here and throughout the paper, when doing matrix computation, we consider $u$ and $x$ as column-vectors,  also we use $u_{x^i}$ or $u^j_{x^i}$ for partial derivatives.

Rewriting \eqref{eq:bols3.1} as 
\begin{equation}
\label{eq:bols3.2'}
{\left( \frac{\partial u}{\partial x}\right)} {\Lfirst} = L{\left( \frac{\partial u}{\partial x}\right)} 
\end{equation} 
we see that the columns $u_{x^i}$ of ${\left( \frac{\partial u}{\partial x}\right)}$ satisfy the equations $L u_{x^i} = u_{x^{i - 1}}$ or equivalently
\begin{equation}
\label{eq:bols3.2}
u_{x^{n-k}} = L^k u_{x^n}, \quad \mbox{where }  L^k = \underbrace{L\cdot L\cdot \ldots \cdot L}_{k \ \mathrm{ times}}, \ k=1,\dots, n-1.
\end{equation}

\begin{Lemma}
\label{lem:bols2.1}
Systems \eqref{eq:bols3.2'} and  \eqref{eq:bols3.2} are equivalent.  In particular, \eqref{eq:bols3.1} is equivalent to \eqref{eq:bols3.2} provided the Jacobi matrix $\left( \frac{\partial u}{\partial x}\right)$ is invertible. 
\end{Lemma}

\begin{proof}  By construction, \eqref{eq:bols3.2} simply means that all the columns of the matrices in the left  and right hand sides of \eqref{eq:bols3.2'} coincides except for the first column.  In other words,  system \eqref{eq:bols3.2'}, as compared to \eqref{eq:bols3.2}, contains one additional vector relation for the first columns of l.h.s. and r.h.s. of \eqref{eq:bols3.2}, namely:
\begin{equation}
\label{eq:22'}
f_1 u_{x^1} + f_2 u_{x^2} + \dots + f_n u_{x^n} = L u_{x^1}
\end{equation}

We need to show that this relation follows from  \eqref{eq:bols3.2}.  This is an easy corollary of the Cayley--Hamilton theorem. Indeed,    substituting $u_{x^{n-k}} = L^k u_{x^n}$ into  \eqref{eq:22'} gives
$$
f_1 L^{n-1} u_{x^n} + f_2 L^{n-2} u_{x^n} + \dots + f_n u_{x^n} = L^n u_{x^n}
$$
or equivalently
$$
\bigl( L^n - f_1 L^{n-1}  - f_2 L^{n-2}  - \dots - f_n \operatorname{Id}  \bigr) u_{x^n} = \chi_{L} (L) u_{x^n} = 0,
$$
which holds true automatically by the Cayley--Hamilton theorem. \end{proof}


Similarly,   to bring $L$ to the second companion form, we need to find an invertible transformation $u=u(x)$ such that 
\begin{equation}
\label{eq:bols3.3}
{\left( \frac{\partial u}{\partial x}\right)}^{-1} L(u) \, {\left( \frac{\partial u}{\partial x}\right)} = {\Lsecond} (x),
\end{equation}
where ${\Lsecond}$ is given by \eqref{second}. Proceeding in a similar way as above, we get $L{\left( \frac{\partial u}{\partial x}\right)} = {\left( \frac{\partial u}{\partial x}\right)} {\Lsecond}$. This gives the following relation on the columns of the Jacobi matrix: $L u_{x_i} = u_{x_{i - 1}} - f_{n-i} u_{x^n}$. For $i = n$ we get $L u_{x^n} = u_{x^{n - 1}} + f_1 u_{x^{n - 1}}$, which yields 
$$
u_{x^{n - 1}}=M_1 u_{x^n}\quad \mbox{with } M_1=L - f_1 {\cdot} \operatorname{Id}.
$$ 
Next for $i=n-1$,  we get $L u_{x^{n-1}}= u_{x^{n - 2}} + f_2 u_{x^n}$, yielding  
$$
u_{x^{n - 2}} = M_2 u_{x^n}\quad \mbox{with } M_2 = LM_1 - f_2{\cdot}\operatorname{Id}
$$ 
and so on. Finally we come to the following system of PDEs:
\begin{equation}
\label{eq:bols3.4}
u_{x^{n-k}} = M_k u_{x^n}, \quad \mbox{where } 
\begin{array}{l}
     M_1 = L - f_1 {\cdot} \operatorname{Id}, \\
    M_k = L M_{k - 1} - f_{k} {\cdot} \operatorname{Id}, \quad 2 \leq k \leq n,
\end{array}
\end{equation}
where $f_1,\dots, f_n$ are the coefficients of the characteristic polynomial of $L$.  Equivalently,
\begin{equation}
\label{eq:bols3.5}
M_k = L^k - f_1 L^{k-1} - f_2 L^{k-2} - \dots - f_{k-1} L - f_k \operatorname{Id},  \quad k=1,\dots, n-1.
\end{equation}    
This system is equivalent to \eqref{eq:bols3.3}, cf. Lemma \ref{lem:bols2.1}.

Thus, we see that reducing $L$ to the both first and second companion forms amounts to solving a quasilinear system of PDEs of the form 
\begin{equation}\label{main_eq}
    u_{x^{n - k}} = A_k(u) u_{x^n}, \quad 1 \leq k \leq n - 1,
\end{equation}
where for the first companion form we set 
$A_k = L^k$,  
while for the second companion form, $A_k = M_k$  with  $M_k$ given by \eqref{eq:bols3.4} or \eqref{eq:bols3.5}.

Notice that \eqref{main_eq} is overdetermined and, in general, not necessarily consistent. However the conditions under which local solutions exist for all initial  data (in other words, the system is in involution) are well-known.

\begin{Proposition}
\label{prop:3.1}
The following properties of \eqref{main_eq} are equivalent
\begin{itemize}
\item[{\rm (A)}] For any real analytic initial condition
$$
\begin{array}{c}
u^1(0, \dots, 0, x^n) = h^1(x^n),\\
u^2(0, \dots, 0, x^n) = h^2(x^n),\\
\dots \\
u^n(0, \dots, 0, x^n) = h^n(x^n), 
\end{array}  \quad \mbox{or shortly}\quad 
u(0, \dots, 0, x^n) = h(x^n),
$$
where $h$ is a real analytic vector-function of one variable, there exists a unique real analytic solution $u=u(x)$ of system \eqref{main_eq}.

\item[{\rm (B)}] Operators $A_k$'s pairwise commute {\rm(}i.e., $A_k A_j = A_jA_k${\rm)}  and  
 \begin{equation}
 \label{eq:propB}
\langle A_k, A_j\rangle (\xi,\xi) \overset{\mathrm{def}}{=}  [A_k\xi, A_j\xi]  - A_j[A_k\xi,\xi] - A_k[\xi,A_j\xi] = 0
\end{equation}
for any vector field $\xi$ and $k,j=1,\dots,n-1$.
\end{itemize}
\end{Proposition}

\begin{proof} The existence of solutions of \eqref{main_eq} for all initial conditions in a more general case is discussed in \cite{Nijenhuis3} (and, in fact, can be derived from the Cartan-K\"ahler theorem \cite{Gold}). The necessary and sufficient condition is $D_{x^{n - i}} \bigl(A_j u_{x^n}\bigr) = D_{x^{n - j}} \bigl(A_i u_{x^n}\bigr)$ on $U(\mathsf{p})$, where $D_{x^k}$ stands for the derivative in virtue of \eqref{main_eq}.  For quasilinear systems this calculation is well-known (see \cite{shar}, \cite{mag2}) and leads to $\rm (B)$. 
\end{proof}

We now apply this Proposition in our special case.

\begin{Proposition} 
\label{prop:3.2}
Both systems \eqref{eq:bols3.2} and \eqref{eq:bols3.4} satisfy Property {\rm (B)}, and therefore, Property {\rm (A)} from Proposition \ref{prop:3.1}.  
\end{Proposition}

\begin{proof}
Although verification of (B) for operators $A_k=L^k$ and $A_k=M_k$ is a nice exercise in tensor calculus, we prefer to make use of an elegant theory of bidifferential ideals introduced by F.\,Magri in \cite{mag1} and then developed by F.\,Magri and P.\,Lorenzoni in \cite{mag2}, in particular, to construct hierarchies of commuting flows of hydrodynamic type. 
They are defined recursively by setting (cf. \eqref{eq:bols3.4})
$$
A_0 = \Id,    \quad  A_{k} = A_{k-1} L - a_k \Id, \  k =1,2,\dots
$$  
for any chain of functions $a_1, a_2, \dots$ satisfying relations  
\begin{equation}
\label{eq:12}
\ddd a_{k+1} = L^* \ddd a_k  -  a_k\ddd a_1.
\end{equation}
Under these conditions,  the operators  $A_k$ generate commuting flows \cite[Proposition 2]{mag2}, i.e. satisfy (B). 

Our situation is just a particular case of this construction.  Indeed, setting  $a_k = 0$,  we obtain the sequence of operators $A_k=L^k$. Hence Property (B) holds for \eqref{eq:bols3.2}.  Of course, this fact is easy to check independently.   

In the case of system \eqref{eq:bols3.4} we only need to check that the coefficients $f_k$ of the characteristic polynomial of $L$ satisfy \eqref{eq:12}  (we may formally set $f_k = 0$ for $k>n$), but these are exactly relations from \cite[Propostion 2.2]{Nijenhuis1}.  Hence Property (B) holds for \eqref{eq:bols3.4}. It is worth noticing that  \eqref{eq:bols3.4} can also be understood as an $\varepsilon$-system in the sense of M.\,Pavlov   \cite{Pavlov} for $\varepsilon=-1$.
\end{proof}

We have just shown that PDE systems \eqref{eq:bols3.2} and \eqref{eq:bols3.4} are both in involution and their (local) solutions $u(x)$ are parametrised by $n$ functions of one variable  (initial conditions $h^1(x^n),\dots, h^n(x^n))$.  To make sure that such a solution $u(x)$ defines a desired coordinate transformation, we need to check  that the Jacobi matrix $\left( \frac{\partial u}{\partial x}\right)$ is non-degenerate at least at the initial point.  Almost all solutions satisfy this property due to $\gl$-regularity of $L$ (moreover this condition is necessary).  

Indeed, for system \eqref{eq:bols3.2}, choose  the initial condition $u(0,\dots,0,x_n)=h(x_n)$ in such a way that the vector $\xi = u_{x^n}(0) = h_{x^n}(0)$  is such that $L^{n-1}\xi, \dots,  L\xi, \xi$  are linearly independent. Since $L$ is $\gl$-regular,  almost all vectors $\xi$ satisfy this condition.  Due to \eqref{eq:bols3.2}, they form the columns of the Jacobi matrix $\left( \frac{\partial u}{\partial x}\right)$ at the initial point $x=(0,\dots,0,0)$.   Hence, at this point $\det\left( \frac{\partial u}{\partial x}\right) \ne 0$ as required.

The same conclusion for solutions of system \eqref{eq:bols3.4} immediately follows from the fact that $\mathrm{Span}(M_{n-1}\xi, \dots, M_1\xi,\xi)= \mathrm{Span}(L^{n-1}\xi, \dots, L\xi,\xi)$.  This completes the proof of Theorem \ref{main:1}.

We see from this proof that reducibility of $L$ to companion forms \eqref{first} and \eqref{second} follows from the involutivity (Property (B) from Proposition \ref{prop:3.1}) of  PDE systems \eqref{eq:bols3.2} and \eqref{eq:bols3.4} respectively.  This property, in turn, follows from the fact that $L$ is Nijenhuis.  It is natural to ask if the latter condition is also necessary for \eqref{eq:bols3.2} and \eqref{eq:bols3.4} to be in involution. The answer is positive under the additional assumption that $L$ is $\gl$-regular.

\begin{Proposition}\label{pro1}   Let $n=\dim \mathsf M >2$ and $L$ be $\gl$-regular. 
\begin{enumerate}
\item If $\langle L^i, L^j \rangle = 0$ for $1 \leq i < j \leq n-1$, i.e., \eqref{eq:bols3.2} is in involution, then $L$ is a Nijenhuis operator.  
\item If $\langle M_i, M_j \rangle = 0$ for $1 \leq i < j \leq n-1$, where $M_i$ is defined as in \eqref{eq:bols3.4},  i.e., \eqref{eq:bols3.4} is in involution, then $L$ is a Nijenhuis operator.
\end{enumerate}
\end{Proposition}
\begin{proof} 1. It is easily seen that for any three commuting operators $L$, $A$ and $B$ the following (algebraic) identity holds:
\begin{equation}\label{bols1}
      \mathcal N_L (A\xi, B\xi)= \bigl(\langle LA, LB \rangle  - L \langle LA, B \rangle  - L \langle A, LB \rangle  + L^2 \langle A, B\rangle \bigr)(\xi, \xi) 
   \end{equation}
Hence, for $A= L^i$, $B=L^j$,  $0\le i,j  <n-1$, we have 
$$
  \mathcal N_L (L^i\xi, L^j\xi)= \bigl(\langle L^{i+1}, L^{j+1} \rangle  - L \langle L^{i+1}, L^j \rangle  - L \langle L^{i}, L^{j+1} \rangle  + L^2 \langle L^{i}, L^j\rangle \bigr)(\xi, \xi) = 0
$$
for any $\xi$. Replacing $\xi$ with $\eta = \xi + L\xi$ and setting $j=n-2$ in this formula, we get
$$
\begin{aligned}
0 &= \mathcal N_L (L^i (\xi + L\xi), L^{n-2}(\xi + L\xi))=  \mathcal N_L (L^i \xi , L^{n-2}\xi)
+\mathcal N_L (L^i (L\xi) , L^{n-2}(L\xi))     \\
&+\mathcal N_L (L^{i+1} \xi, L^{n-2}\xi)+
\mathcal N_L (L^i \xi, L^{n-1}\xi) = 0 + 0 + 0 + \mathcal N_L (L^i \xi, L^{n-1}\xi).
  \end{aligned}
  $$
Thus, $\mathcal N_L$ vanishes for any pair of vectors from the set $\xi, L\xi, \dots, L^{n - 1}\xi$.
As $L$ is $\gl$-regular, one can choose $\xi$ in a way that $\xi, L\xi, \dots, L^{n - 1}\xi$ form a basis in the tangent space. Hence,
$\mathcal N_L = 0$, as stated.

2. In what follows we assume that $M_0=\Id$ and $M_n = 0$ which perfectly agrees with the above definition of $M_i$'s (due to the Cayley-Hamilton theorem). We start with
\begin{Lemma}\label{lm1}
If  $\langle M_i, M_j \rangle = 0$ for $1 \leq i < j \leq n-1$, then the following identities hold
\begin{equation}\label{pro2_1}
   \ddd f_{j+1} (M_i\xi) - \ddd f_{i + 1} (M_j\xi) =0, \quad i,j = 0, \dots, n - 1.
\end{equation}
where $f_i$ are coefficients of the characteristic polynomial of $L$ and $\xi$ is an arbitrary tangent vector.
\end{Lemma}

\begin{proof} In formula \eqref{eq:propB}, the expression $\langle A, B\rangle$ is treated as a (vector-valued) quadratic form on the tangent bundle (one assumes that $A$ and $B$ commute).  We can also naturally interpret it as a symmetric bilinear form by setting:
$$
\langle A, B\rangle (\xi,\eta) = \frac{1}{2}\bigl([A\xi, B\eta] - A[\xi, B\eta] - B[A\xi,\eta] + [A\eta, B\xi] - A[\eta, B\xi] - B[A\eta,\xi]\bigr)
$$
Obviously  $\langle A, B\rangle(\xi,\xi)\equiv 0$ implies $\langle A, B\rangle(\xi,\eta)\equiv 0$.

First we observe the following (purely algebraic) identity:
$$
    \langle M_i L, M_j \rangle (\xi, \xi) + \langle M_i, M_jL \rangle (\xi, \xi)  = M_i\langle L, M_j \rangle (\xi, \xi) - M_j \langle L, M_i \rangle (\xi, \xi) - 2\langle M_j, M_i \rangle (L\xi, \xi).
$$
In our case we have  $L=M_1 +f_1\Id$ and, in addition,  $\langle M_i, M_j \rangle \equiv 0$, which gives:
\begin{equation}\label{bols4}
\begin{aligned}
    \langle M_i L, M_j \rangle  + \langle M_i, M_jL \rangle   &= M_i\langle L, M_j \rangle  - M_j \langle L, M_i \rangle = \\
    M_i\langle M_1 + f_1\Id, M_j \rangle  &- M_j \langle M_1+f_1\Id, M_i \rangle = M_i\langle f_1\Id, M_j \rangle  - M_j \langle f_1\Id, M_i \rangle 
 \end{aligned}
    \end{equation}

Using \eqref{bols4} and the definition of $M_i$'s, we now compute the right hand side of the identity $0=\langle M_{i + 1}, M_j \rangle  + \langle M_{i}, M_{j + 1} \rangle$: 
\begin{equation}\label{k1}
    \begin{aligned}
        0 & = \langle M_{i + 1}, M_j \rangle  + \langle M_{i}, M_{j + 1} \rangle  =  \langle LM_{i}- f_{i+1}\Id, M_j \rangle  + \langle M_{i}, L M_j -f_{j+1}\Id\rangle\\
        & = \langle LM_{i}, M_j \rangle  + \langle M_{i}, L M_j \rangle  - \langle f_{i + 1} \operatorname{Id}, M_j \rangle  - \langle M_{i}, f_{j + 1} \operatorname{Id} \rangle  = \\ 
       & = M_i \langle f_1 \operatorname{Id}, M_j \rangle - M_j \langle f_1 \operatorname{Id}, M_i \rangle - \langle f_{i + 1} \operatorname{Id}, M_j \rangle  + \langle f_{j + 1} \operatorname{Id}, M_i \rangle .\\
    \end{aligned}
\end{equation}
Notice that 
\begin{equation}\label{k7}
    \langle f \operatorname{Id}, A \rangle (\xi, \xi) = [f\xi, A\xi] - f[\xi, A\xi] - A[f\xi, \xi] = \ddd f(A\xi) \xi - \ddd f(\xi) A\xi
\end{equation}
for an arbitrary function $f$ and operator $A$. Applying this relation to \eqref{k1} gives
\begin{equation}\label{k2}
    \begin{aligned}
        0 & = M_i \langle f_1 \operatorname{Id}, M_j \rangle (\xi,\xi)- M_j \langle f_1 \operatorname{Id}, M_i \rangle(\xi,\xi) - \langle f_{i + 1} \operatorname{Id}, M_j \rangle (\xi, \xi) + \langle f_{j + 1} \operatorname{Id}, M_i \rangle (\xi, \xi) = \\
        & = \big( \ddd f_1(M_j\xi) - \ddd f_{j + 1}(\xi)\big) M_i\xi - \big( \ddd f_1(M_i\xi) - \ddd f_{i + 1}(\xi)\big) M_j\xi + \\
        & + \big( \ddd f_{j + 1} (M_i \xi) - \ddd f_{i + 1} (M_j \xi)\big)\xi.
    \end{aligned}
\end{equation}
Recall that $L$ is $\gl$-regular. Hence  $\xi, L\xi, \dots, L^{n - 1}\xi$ are linearly independent for almost all tangent vectors $\xi$. By formula \eqref{eq:bols3.4}    for $M_i$, this is still true for $\xi, M_1\xi, \dots, M_{n - 1}\xi$. Therefore  $\xi, M_i\xi, M_j\xi$ are linearly independent in \eqref{k2}, and the coefficients of this linear combination vanish. Thus,  $\ddd f_{j+1} (M_i\xi) - \ddd f_{i + 1} (M_j\xi) =0$ for {\it almost all} vectors $\xi$ and by continuity for {\it all} vectors. Lemma is proved.
\end{proof}
 

Similar to the first case, our goal is to show that $\mathcal N_L (M_i\xi, M_j\xi)=0$ for all $i,j=0,\dots,n-1$. Since $M_0\xi, M_1\xi, \dots, M_{n-1}\xi$  form a basis for a generic vector $\xi$,  this will imply $\mathcal N_L = 0$. 

As above we use \eqref{bols1} with $A=M_i$ and $B=M_j$:
\begin{equation}
\label{eq:bols1'}
\mathcal N_L(M_i\xi, M_j\xi)=\bigl(\langle LM_i, LM_j \rangle  - L \langle LM_i, M_j \rangle  - L \langle M_i, LM_j \rangle  + L^2 \langle M_i, M_j\rangle\bigr)(\xi,\xi). 
\end{equation}
Substituting  $LM_i = M_{i+1} + f_{i+1}\Id$ and using the relations $\langle M_i,M_j\rangle = 0$ ($ i,j=0,\dots,n$) and identity $\langle f_{i+1}\Id ,f_{j+1}\Id\rangle = 0$, we can rewrite the vector-valued quadratic form in the right hand side of \eqref{eq:bols1'} as follows:
\begin{equation}\label{bols2}
\begin{aligned}
     & \langle LM_i, LM_j \rangle  - L \langle LM_i, M_j \rangle  - L \langle M_i, LM_j \rangle  + L^2 \langle M_i, M_j\rangle = \\
      &\langle M_{i+1} +f_{i+1}\Id , M_{j+1}+ f_{j+1}\Id\rangle   
      - L  \langle M_{i+1} +f_{i+1}\Id, M_j \rangle  - L \langle M_i, M_{j+1} +f_{j+1}\Id \rangle  
       =\\
      &\langle f_{i+1}\Id , M_{j+1}\rangle   + \langle M_{i+1} ,  f_{j+1}\Id\rangle   
      - L  \langle f_{i+1}\Id, M_j \rangle  - L \langle M_i, f_{j+1}\Id \rangle  = \\
      &\langle f_{i+1}\Id , L M_{j} - f_{j+1}\Id \rangle   + \langle L M_{i} - f_{i+1}\Id ,  f_{j+1}\Id\rangle   
      - L  \langle f_{i+1}\Id, M_j \rangle  - L \langle M_i, f_{j+1} \Id \rangle  = \\
      &\langle f_{i+1}\Id , L M_{j} \rangle   + \langle L M_{i} ,  f_{j+1}\Id\rangle   
      - L  \langle f_{i+1}\Id, M_j \rangle  - L \langle M_i, f_{j+1} \Id \rangle
\end{aligned}
   \end{equation}
Hence, using \eqref{k7}, we get:
\begin{equation} \label{bols3}
\begin{aligned}
 \mathcal N_L (M_i\xi&, M_j\xi)= \bigl( \langle f_{i+1}\Id , LM_{j} \rangle   + \langle LM_{i} ,  f_{j+1}\Id\rangle   
      - L  \langle f_{i+1}\Id, M_j \rangle  - L \langle M_i, f_{j+1} \Id \rangle \bigr) (\xi,\xi) \\
&\,= \dd f_{i+1} (LM_j \xi) \,\xi - \dd f_{i+1}(\xi) \,L M_j \xi  -  \dd f_{j+1} (LM_i \xi)\, \xi - \dd f_{j+1}(\xi)\, L M_i \xi  \\
&\,- L \bigl(\dd f_{i+1} (M_j\xi)\, \xi - \dd f_{i+1}(\xi)\, M_j\xi\bigr) + L \bigl(\dd f_{j+1} (M_i\xi) \,\xi - \dd f_{j+1}(\xi)\, M_i\xi\bigr) \\ 
&\,=   \bigl(\dd f_{i+1} ( M_j (L\xi))-\dd f_{j+1} ( M_{i}(L\xi)) \bigr)\,\xi 
- \bigl( \dd f_{i+1} ( M_{j}\xi) - \dd f_{j+1} ( M_{i}\xi) \bigr) \,  L\xi.
  \end{aligned}
\end{equation}
It remains to notice that  the coefficients in front of $\xi$ and $L\xi$ vanish by Lemma \ref{lm1}, which completes the proof. \end{proof}


\begin{Remark}{\rm The $\gl$-regularity assumption in Proposition \ref{pro1}  is essential.  Indeed,  consider an operator $L$ such that $L^2 = \operatorname{Id}$ or $L^2 = 0$. Then the involutivity conditions  $\langle L^i,  L^j\rangle=0$ and $\langle M_i,  M_j\rangle=0$ obviously hold.  However, $L$ does not need to be Nijenhuis. 
}\end{Remark}


\section{Proof of Theorem \ref{main:2}}\label{sect:proof2}

The goal of this section is to study and solve the PDE system:
 \begin{equation}\label{first_set_again}
    \begin{aligned}
        & \frac {\partial f_i} {\partial x^{j}} = f_i  \frac{\partial f_1}{\partial x^{j+1}} + \frac{\partial f_{i + 1}} {\partial x^{j+1}} ,  \\
        & \frac{\partial f_n}{\partial x^{j}}  = f_n \frac{\partial f_1}{\partial x^{j+1}}.
        \end{aligned}
    \end{equation}
$1 \leq i, j \leq n-1$. According to Theorem \ref{main:1}, every collection of functions $f_i$ satisfying this system defines a  $\gl$-regular Nijenhuis operator  
of the form
\begin{equation}\label{first_again}
L(x) =  \Lfirst (x) =  \left(\begin{array}{ccccc}
     f_1 & 1 & 0 & \dots & 0  \\
     f_2 & 0 & 1 & \ddots & \vdots \\
     \vdots & \vdots & \ddots & \ddots & 0 \\
     f_{n - 1} & 0 & \dots & 0 & 1 \\
     f_n & 0 & \dots &  0 & 0 \\
    \end{array}\right),
    \end{equation}
and vice versa,  if this operator is Nijenhuis, then these functions satisfy  \eqref{first_set_again}. Throughout this section we deal with  a Nijenhuis operator written in first companion form and use $L$ instead of $\Lfirst$ to simplify notation. 

If we denote  $f=(f_1,\dots, f_n)$  (in matrix form, we think of $f$ as a {\it column}-vector), then   \eqref{first_set_again}  simply means that
\begin{equation}
\label{eq:20}
f_{x^j} = L f_{x^{j+1}} , \quad \mbox{or equivalently} \quad    f_{x^j} = L^{n-j} f_{x^n}, \quad  1\le j \le n-1. 
\end{equation}

Observe that \eqref{eq:20} coincides with the PDE system  \eqref{main_eq} with $A_i = L^i$ that we used above to reduce $L$ to the first companion form.  The difference is that now the operator $L$ is already in companion form so \eqref{eq:20}  (equivalently \eqref{first_set_again}) defines transformations $f = f(x)$ that preserve this form. (Notice that we are now interested in both invertible and non-invertible transformations.)  This observation can be rephrased as follows.

System \eqref{first_set_again} (or equivalently, \eqref{eq:20}) can be written in the following matrix form (see Lemma \ref{lem:bols2.1}):
\begin{equation}
\label{eq:21}
\left(\frac{\partial f}{\partial x} \right)  L    =  L  \left(\frac{\partial f}{\partial x}\right) 
\end{equation}

If $f=f(x)$ defines an invertible transformation, i.e.,  the Jacobi matrix $ \left(\frac{\partial f}{\partial x} \right)$ is invertible, then  
\eqref{first_set_again} is equivalent to 
\begin{equation}
\label{eq:23}
\left(\frac{\partial x}{\partial f} \right)  L    =  L  \left(\frac{\partial x}{\partial f}\right).
\end{equation}
which is a linear system of PDEs for unknown functions $x^i=x^i(f)$.   This system can be easily solved in a neighbourhood of any point $f_1=c_1,\dots, f_n=c_n$ for any initial condition and the corresponding solution can be found without integration.  We are grateful to E.Ferapontov for explaining us the idea of this method.

\begin{Proposition}\label{prop:33}
For any real-analytic initial condition 
\begin{equation}
\label{eq:25}
\begin{aligned}
x^1(c_1,\dots,c_{n-1}, c_n + \tau) &= v_1(\tau)\\
& \dots \\
x^n(c_1,\dots,c_{n-1}, c_n + \tau) &= v_n(\tau)\\
\end{aligned}
\end{equation}
where $\tau$ belongs to a neignborhood of zero,
there exists a unique (local) real analytic solution of \eqref{eq:23}. This solution can be found by using the following procedure.
Consider a real analytic function $F(t)$ constructed from $v_1,\dots,v_n$ as follows:
\begin{equation}
\label{eq:25'}
F(t) = v_1 \bigl(p(t)\bigr) + t v_2 \bigl(p(t)\bigr) + \dots + t^{n-1} v_n \bigl(p(t)\bigr), 
\end{equation}
where  $p(t) = t^n - c_1 t^{n-1} - \dots - c_{n-1} t - c_n$.
Then the solution $x(f)$ satisfying the initial conditions \eqref{eq:25} takes the form
\begin{equation}
\label{eq:26}
x(f) = F(L) e_n, \quad  \mbox{where}   \  e_n = \begin{pmatrix}  0 \\ \vdots \\ 0 \\ 1   \end{pmatrix},
\end{equation}
i.e., $x(f)$ is the last column of the matrix $F(L)$. 
\end{Proposition}

\begin{proof}
We first notice that the function \eqref{eq:25'} is real analytic in a neighbourhood of the spectrum of $L_0=L(c_1,\dots,c_n)$ \footnote{In a neighbourhood of a pair of complex conjugate eigenvalues  $\lambda, \bar\lambda$,  {\it real} analyticity of $F$ means, in addition,  that in $F(\overline{z}) =\overline{F(z)}$.}  and therefore  locally $F(L)$ is a real analytic matrix function (see \cite{higham}).

To prove formula \eqref{eq:26}, it is sufficient to check two facts:    

\begin{itemize}
\item Formula \eqref{eq:26}   gives a solution of \eqref{eq:23}   for an arbitrary polynomial  $F(t)$  (then this will be true for any real-analytic functions by continuity as polynomials are everywhere dense in this space).
\item The solution defined by \eqref{eq:26} indeed satisfies the required initial conditions.
\end{itemize}

Since the PDE system \eqref{eq:26} is linear,  instead of an arbitrary polynomial $F(t)$ it is sufficient to consider only polynomials of the form $t^k$, $k=0,1,2,\dots$.   We will check this fact by induction, namely, we prove the following 
\begin{Lemma}
Let $x(f) = \begin{pmatrix} x^1(f) \\ \vdots \\ x^n(f) \end{pmatrix}$ be a solution of \eqref{eq:23}, then  $\tilde x(f) = L x(f)$ is a solution also. In other words, multiplication by $L$ sends solutions to solutions. 
\end{Lemma}

\begin{proof}    It is easy to see that the Jacobi matrix $\left(  \frac{\partial \tilde x}{\partial f}  \right) $ takes the form
$$
\left(  \frac{\partial \tilde x}{\partial f}  \right)  =  L  \left(  \frac{\partial x}{\partial f}  \right) + x^1 \cdot \mathrm{Id}
$$
Hence, if $\left(  \frac{\partial x}{\partial f}  \right)$  commutes with $L$, then $\left(  \frac{\partial \tilde x}{\partial f}  \right)$ commutes also, i.e. $\tilde x(f)$ is a solution  of \eqref{eq:23}, as stated.
\end{proof}

It is easy to see that any constant vector-function $x(f) = a\in\R^n$, and in particular $x(f)=e_n$, is a solution of \eqref{eq:23}, then by induction, we have $Le_n$, $L^2e_n$, \dots, $L^k e_n$ are all solutions too, implying that $x(f) = F(L) e_n$ is a solution of  \eqref{eq:23} for any polynomial and hence for any real analytic function $F$. 

To check the  initial conditions, we compute $F(L)e_n$, i.e., the last column of $F(L)$  for  $f_i = c_i$, ($i=1,\dots,n-1$), $f_n=c_n+\tau$.

First we notice that the substitution of $L$ (with the above indicated values of $f_i$) into the polynomial $p(t) = t^n - c_1t^{n-1} -\dots - c_{n-1} t - c_n$ gives
$$
p(L) = \tau \cdot \operatorname{Id}.
$$  
Indeed,  
$$
\chi_{L} (t) =  t^n - f_1t^{n-1} -\dots - f_{n-1} t - f_n =  t^n - c_1t^{n-1} -\dots - c_{n-1} t - (c_n + \tau) = p(t) - \tau.
$$
Hence, $p(t) = \chi_{L} (t) +\tau$ and $p(L) = \chi_{L} (L) +\tau\cdot\Id = 0 + \tau\cdot\Id $, as stated.

Next,  
$$
F(L) = \sum v_{n-k} (p(L)) (L)^k = \sum v_{n-k} (\tau\cdot \operatorname{Id}) L^k =  \sum v_{n-k} (\tau)L^k.
$$
It remains to notice that the last column of $L^k$ is the $(n-k)$-th basis vector $e_{n-k}$ and therefore the last column of $F(L)$ is $\sum v_{n-k} (\tau) e_{n-k} = \begin{pmatrix}  v_1(\tau) \\ \vdots \\ v_n(\tau) \end{pmatrix}$.   \end{proof}

It is worth mentioning that the formulas for solutions of \eqref{eq:23} are based on a more general phenomenon (see \cite{Tsarev}) related to quasilinear systems of PDEs of type \eqref{main_eq} which, in the case of Nijenhus operators, can be explained as follows.   Consider the operator 
$$
M = x_1 L^{n-1}(u) + x_2 L^{n-2}(u) + \dots + x_{n-1} L(u) + x_n\operatorname{Id}
$$ 
and algebraic relation 
\begin{equation}
\label{eq:30}
M = F(L),
\end{equation}
where  $F$ is an analytic matrix function.   If $L$ is a $\gl$-regular Nijenhuis operator in coordinates $u^1,\dots, u^n$, then $x_i$ can be expressed in terms of $u$ and the inverse function $u(x)$ (if it exists!) is a solution of the PDE system
$$
u_{x^{j}} = L^{n-j} u_{x^{n}},   \quad j=,\dots, n-1.
$$

Our formula is just a particular version of this fact for $L$ being in the first companion form. In this case, one can resolve \eqref{eq:30}  (this is by the way another advantage of the first companion form)  explicitly and choose an appropriate matrix function $F$ for a prescribed initial condition.

Proposition \ref{prop:33} describes all the solutions $x=x(f)$ of \eqref{eq:23} and therefore all invertible solutions $f=f(x)$ of \eqref{first_set_again}.  In other words, we obtain description of all Nijenhuis operators in companion form which are differentially non-degenerate  (see \cite[Definition 2.10]{Nijenhuis1}  and Remark \ref{rem:1.4}).  Notice that all these operators can be transformed to each other by a suitable coordinate change and, in particular, each of them can be brought to the form \eqref{eq:nilp_pert}.  Equivalently, we can say that such operators form the {\it largest} (or, {\it generic}) orbit of the groupoid that consists of coordinate transformations acting on Nijenhuis operators of type \eqref{first}. 


However,  our goal is to describe all the solutions $f=f(x)$ of \eqref{first_set_again}, both invertible and non-invertible.  Moreover, we would like to be able to construct the solution that corresponds to prescribed initial conditions $f(0,\dots,0,x^n) = v(x^n)$. The above method does not allow us to do this and we need to modify it.  This is exactly what Theorem \ref{main:2} does by replacing \eqref{eq:30} with a more general algebraic relation of the form $r(L,M)=0$ which, in some sense, interchanges the roles of $L$ and $M$ and, as a result, $x$ and $f$.  

We now prove Theorem \ref{main:2}.  As above,
we set $M = x^1 L^{n-1} + x^2 L^{n-2} + \dots + x^{n-1} L + x^n\operatorname{Id}$ and consider the matrix function
$$
 r(L,M) = L^n  -  v_1(M) L^{n-1} - v_2(M) L^{n-2} - \dots - v_{n-1}(M) L - v_n(M)
 $$ 
 where  $v_i(t)$ are the functions defining the initial conditions for  \eqref{first_set}   (or equivalently, \eqref{eq:20}).  
 
We need to show that the solution $f(x)=(f_1(x),\dots, f_n(x))$ of \eqref{first_set}  with prescribed initial conditions can be obtained by resolving the relation  $r(L,M)=0$ with respect to the coefficients of the characteristic polynomial of $L$.
 
We first notice that this relation is invariant in algebraic sense so that we may consider the matrices $L$ and $M$ in any basis we like. Of course, we will assume that $L$ is written in companion form \eqref{first_again}. 

The matrix $\sum v_i(M) L^{n-i}$ commutes with $L$ and its entries are analytic functions in $x$ and $f$. This matrix can be uniquely presented as linear combination
\begin{equation}\label{r1}
\sum v_i(M) L^{n-i} = g_1 L^{n-1} + \dots + g_{n-1} L + g_{n} \mathrm{Id},   
\end{equation}
where $g_i=g_i(x,f)$ are noting else but the entries of the last column of $\sum v_i(M) L^{n-i}$   (this easily follow from the fact that $L$ is a companion matrix).  Thus, the relation $r(L,M)=0$  reads
$$
L^n - g_1 L^{n-1} - \dots - g_{n-1} L - g_{n} \mathrm{Id} =0.
$$ 
Comparing with 
$$
L^n - f_1 L^{n-1} - \dots - f_{n-1} L - f_{n} \mathrm{Id} =0 \quad\mbox{ (Cayley--Hamilton theorem)}
$$
and using $\gl$-regularity of $L$ we come to the system of algebraic relations
$$
f_i = g_i(x,f).
$$
To resolve these relations w.r.t. $f$, i.e. to find $f_i =f_i(x)$ as a real analytic function of $x$  (for small $x$), it is sufficient to check that  $\frac{\partial g_i}{\partial f_\alpha} (0,\dots,0,  x^n, f)= 0$, which is obviously true as 
$$
\sum v_i(0\cdot L^{n-1}  + \dots + 0\cdot L +  x^n{\cdot} \mathrm{Id}) L^{n-i} = \sum v_i(x^n) L^{n-i},
$$ 
implying that  $g_i(0,\dots,0,  x^n, f)$ coincides with $v_i(x^n)$ and therefore does not depend on $f_\alpha$.   This proves the first statement and also shows that the initial conditions are indeed fulfilled: if $x^1=\dots=x^{n-1}=0$, then  $f_i(0,\dots, 0, x^n)=g_i(0,\dots, 0,x^n,f)=v_i(x^n)$, as required.

The last step is to show that  $f_i(x)$'s so obtained satisfy \eqref{first_set} or equivalently \eqref{eq:20}. We start with two lemmas concerning $g(x, f)$.

\begin{Lemma}\label{lem1}
The vector-function $g = (g_1\, \dots \, g_n)^\top$ satisfies \eqref{eq:20}, i.e.,
\begin{equation*}
g_{x^j} = L  \, g_{x^{j+1}},   \quad j=1,\dots, n-1.
\end{equation*}
\end{Lemma}

\begin{proof}

Differentiating \eqref{r1} w.r.t. $x^j$ we get
\begin{equation}\label{r2}
L^{n-j}(v'_1(M) L^{n - 1} + \dots + v'_n (M) \operatorname{Id}) = \pd{g_1}{x^j} L^{n - 1} + \dots + \pd{g_n}{x^j} \operatorname{Id}.
\end{equation}
Similarly, differentiating \eqref{r1} w.r.t $x^{j+1}$  we get 
\begin{equation}\label{r3}
    L^{n-j-1} \bigl(v'_1(M) L^{n - 1} + \dots + v'_n (M) \operatorname{Id} \bigr) = \pd{g_1}{x^{j+1}} L^{n - 1} + \dots + \pd{g_n}{x^{j+1}} \operatorname{Id}.
\end{equation}
Comparing \eqref{r3} and \eqref{r2} gives
$$
L \left(\pd{g_1}{x^{j+1}} L^{n - 1} + \dots + \pd{g_n}{x^{j+1}}\right) = \pd{g_1}{x^j} L^{n - 1} + \dots + \pd{g_n}{x^j} 
$$
Applying the Cayley--Hamilton theorem (i.e., $L^n = f_1 L^{n-1} +\dots + f_n$)  we obtain
$$
\begin{aligned}
    \left(\pd{g_2}{x^{j+1}} + f_1 \pd{g_1}{x^{j+1}}\right) L^{n - 1} &+ \dots + \left(\pd{g_n}{x^{j+1}} + f_{n - 1} \pd{g_1}{x^{j+1}}\right) L + f_n \pd{g_1}{x^{j+1}} \\
    &= \pd{g_1}{x^j} L^{n - 1} + \dots + \pd{g_n}{x^j} \operatorname{Id}.
\end{aligned}    
$$
Since $L^{n-1},\dots, L, \mathrm{Id}$ are linearly independent,  we get $g_{x^j} = \Lfirst g_{x^{j+1}}$ which coincides with 
the statement of the lemma as  $L=\Lfirst$.
\end{proof}

\begin{Lemma}\label{lem2'}
Let $\left( \pd{g}{f} \right)$ be the matrix of partial derivatives  $\pd{g_i}{f_j}$, $1\le i,j \le n$. Then we have
\begin{equation}\label{r5}
L \left( \pd{g}{f} \right) = \left( \pd{g}{f} \right) L. 
\end{equation}
\end{Lemma}
\begin{proof}
As already noticed, $g_i=g_i(x,f)$ are the entries of the last column of $\sum v_i(M) L^{n-i}$, if $L$ is written in a companion basis, i.e., $L=\Lfirst$.  On the other hand for fixed $x$, the expression  $\sum v_i(M) L^{n-i}$ can be treated as an analytic function $F(L)$.  After this remark,  \eqref{r5} is just a part of the conclusion Proposition \ref{prop:33} (with $x$  replaced by $g$).
\end{proof}

Consider the implicit equation $f = g(x, f)$. Differentiating it w.r.t. $x^j$ and $x^{j+1}$ we get $f_{x^j} = g_{x^j} + \Bigl(\pd{g}{f}\Bigr) f_{x^j}$ and $f_{x^{j+1}} = g_{x^{j+1}} + \Bigl(\pd{g}{f}\Bigr) f_{x^{j+1}}$. Multiplying by $L$ and subtracting we get
\begin{equation}\label{r7}
f_{x^{j}} - L f_{x^{j+1}} = g_{x^{j}} - L g_{x^{j+1}} + \left(\pd{g}{f}\right) f_{x^j} - L \left( \pd{g}{f} \right) f_{x^{j+1}}    
\end{equation}
Applying  Lemma \ref{lem1} and \eqref{r5},  we see that \eqref{r7} can be written as
$$
\left(\operatorname{Id} - \left(\pd{g}{f}\right)\right) (f_{x^j} - Lf_{x^{j+1}}) = 0.
$$
As already noticed,  the matrix $\Bigl(\pd{g}{f}\Bigr)$ vanishes for $x=(0,\dots,0,x^n)$, therefore locally $\operatorname{Id} - \Bigl(\pd{g}{f}\Bigr)$ is invertible and we conclude that $f$ satisfies \eqref{eq:20} or, equivalently \eqref{first_set}, which completes the proof of Theorem \ref{main:2}.

 \section{Nijenhuis perturbations of a Jordan block}\label{sect:perturb}

Our next goal is to discuss Nijenhuis perturbations of a Jordan block $J_0$, that is, Nijenhuis operators of the form $L(x)= J_0 \,+$ higher order terms.  
Recall that a generic Nijenhuis perturbation of $J_0$ is described by the following 

\begin{Proposition}[\cite{Nijenhuis1}, see also Remark \ref{rem:1.4}]\label{prop:5.1}
Let $L$ be a Nijenhuis operator such that at a point $\mathsf p$,  the operator $L(\mathsf p)$ is similar  to the (nilpotent) Jordan block $J_0$.  Assume that the differentials of the coefficients of the characteristic polynomial of $L$ are linearly independent at $\mathsf p$.   Then in a neighbourhood of $\mathsf p$ there exist local coordinates $x^1, \dots, x^n$ with $\mathsf p \simeq (0,\dots, 0)$ in which $L(x)$ is given by \eqref{eq:nilp_pert}.
\end{Proposition}  

In this case, it is easily seen that at a generic point $\mathsf{q} \in U(\mathsf{p})$, the operator $L(\mathsf{q})$ becomes semisimple with distinct eigenvalues.  Moreover,  for any collection of real and complex conjugate numbers  $S=\{\lambda_1, \dots , \lambda_k, \mu_1, \bar\mu_1,\dots, \mu_s ,\bar\mu_s\}$  $(k+2s=n)$  sufficiently close to zero and not necessarily distinct, there exists a unique point $\mathsf{q}\in U(\mathsf{p})$ such that $S$ is the spectrum of $L(\mathsf{q})$.  In particular, we see that in $U(\mathsf{p})$ we can find operators of all possible algebraic types that are potentially allowed for $\gl$-regular operators  (this means that for repeated eigenvalues there will be only one Jordan block). 

It is natural to ask whether there are other scenarios of Nijenhuis perturbations,  for instance,  with a prescribed algebraic structure of $L$ at a generic point $\mathsf{q}$. For instance, can a Jordan block $J_0$  split into two smaller Jordan blocks of prescribed sizes $k_1$, $k_2$, $k_1+k_2=n$? 

The answer is positive. Let us show that all scenarios are possible.  According to Theorems \ref{main:1} and \ref{main:2}, we may assume that $L(x)=\Lfirst(x)$ where $\Lfirst$ is given by \eqref{first} and the coefficients $f_1(x),\dots, f_n(x)$ of the characteristic polynomial $\chi_{L}$ satisfy \eqref{first_set}. To construct the corresponding perturbation one just needs to make sure that the desired scenario happens on the initial straight line $x(\tau)=(0,\dots,0,\tau)$.  Assume that on this initial line at a generic point $\tau\in (-\varepsilon,\varepsilon)$, the characteristic polynomial  
\begin{equation}
\label{eq:29}
\chi_{L(x(\tau))}(t)= t^n - f_1(x(\tau))t^{n-1} -\dots  - f_n(x(\tau))=t^n -v_1(\tau)t^{n-1} -\dots  - v_n(\tau)
\end{equation}
factorises as follows
$$
\chi_{L(x(\tau))} (t) = (t - \mu_1(\tau))^{k_1} (t - \mu_2(\tau))^{k_2}\cdots (t - \mu_s(\tau))^{k_s}. 
$$
where $\mu_i(\tau)$ are some real analytic functions in $\tau$  (perhaps complex valued).  In other words, at a generic points of the initial line $x(\tau)$, this polynomial has  $s$ distinct roots with multiplicities  $k_1,\dots, k_s$. 

According to Theorem \ref{main:2},  to describe the solution $f=f(x)$  with given initial conditions $f(x(\tau))=v(\tau)$ we should consider the relation
$$
r(L,M) = L^n - v_1(M) L^{n-1} - \dots - v_{n-1}(M) L  - v_n(M)=0,  \quad \mbox{with } \  M = \sum_{i=1}^n x^i L^{n-i}, 
$$
and then ``solve'' it to find the coefficients of the characteristic polynomial of $L$ in terms of $x^1, \dots , x^n$. Notice that $r(L,M)$ is just the polynomial  \eqref{eq:29} after the substitution $\tau\mapsto M$, $t \mapsto L$. We know that this polynomial factorises (for scalars $\tau$ and $t$,  but here the difference between scalars and matrices is not essential), hence we can write
$$
r(L,M) = (L - \mu_1(M))^{k_1} (L - \mu_2(M))^{k_2}\cdots (L - \mu_s(M))^{k_s} =0
$$ 
where  $\mu$ is now treated as an analytic matrix function.

The eigenvalues of $L$  (as functions in $x$) can now be found from relations of the form:
$$
\lambda = \mu_i \left(x_1 \lambda^{n-1} + x_2 \lambda^{n-2} + \dots + x_{n-1}\lambda + x_n\right)
$$
By the Implicit Function Theorem this can be done uniquely 
in a neighbourhood of a point $(0,\dots,0,\tau)$ in such a way that  $\lambda (0,\dots,0, \tau) = \mu_i(\tau)$, as needed. 

No other eigenvalues may occur.  The multiplicities of these eigenvalues will be as expected  since this condition is fulfilled on the initial line. This shows that at a generic point we have
$$
\chi_{L(x)}(t) = (t - \lambda_1(x))^{k_1} (t - \lambda_2(x))^{k_2}\cdots (t - \lambda_s(x))^{k_s} 
$$ 
where $\lambda_i(x)$  will be real analytic functions such that  $\lambda_i(0,\dots,0,\tau) = \mu_i(\tau)$  (these relations hold as soon as $\mu_i(\tau)$ makes sense).

This argument leads us to the following property of the discriminant of the polynomial $\chi_f = t^n - f_1 t^{n-1} - \dots - f_{n-1} t - f_n$.

\begin{Proposition}
Let $f(x) = \bigl(f_1(x),\dots, f_n(x)\bigr)$ be a solution of \eqref{first_set_again}.  Assume that the discriminant $\mathcal D(f_1,\dots, f_n)$ of the polynomial  $\chi_{f(x)}(t)=t^n - f_1(x)t^{n-1} -\dots -f_{n-1}(x) t - f_n(x)$ vanishes on the initial straight line  $x(\tau) = (0, \dots, 0, \tau)$.  Then the discriminant vanishes identically for all $x=(x^1,\dots, x^n)$.

Similarly, for each partition $n=k_1+\dots+k_s$, consider  the algebraic variety  $\overline W_{k_1,\dots,k_s}\subset \R^n(f_1,\dots,f_n)$ that is the Zariski closure of the set $W_{k_1,\dots,k_s}$ of those $f\in \R^n$ for which $\chi_f(t)$ has $s$ distinct roots with multiplicities $k_1,\dots,k_s$.     If  $f(x(\tau)) \in \overline W_{k_1,\dots,k_s}$ for the initial line $x(\tau) = (0, \dots, 0, \tau)$, then $f(x)\in \overline W_{k_1,\dots,k_s}$ for all $x=(x^1,\dots, x^n)$. \end{Proposition}

Clearly, the second part of this proposition immediately implies Theorem \ref{main:4}.

Let us finally discuss an example showing how Theorem \ref{main:2}  works in practice to construct explicit examples of Nijenhuis operators with non-trivial singularities.

\begin{Ex}\label{exam:5.1}{\rm
Consider the three dimensional case and, in the settings of Theorem \ref{main:2},  
define the initial conditions in such a way that on the initial line $x(\tau) = (0,0,\tau)$  the characteristic polynomial of $L$ takes the form
$$
\chi_{L(x(\tau))}(\lambda) = (\lambda - \tau)^2 (\lambda - 2\tau) = \lambda^3 - 4\tau \lambda^2 + 5\tau^2 \lambda - 2\tau^3,
$$
or equivalently
$$
f_1 (0,0,\tau) = 4\tau = v_1(\tau), \quad f_2(0,0,\tau) = -5\tau^2=v_2(\tau), \quad f_3(0,0,\tau)=2\tau^3=v_3(\tau).
$$

The algorithm described in Theorem \ref{main:2}  allows us to reconstruct the functions $f_1, f_2, f_3$.  To that end we need to 
use the matrix relation 
$$
L^3 - (4M) L^2 + (5M^2) L - 2M^3 = 0   \quad \mbox{with } \ M=x_1L^2 + x_2L + x_3\operatorname{Id},
$$
to express the coefficients of the characteristic polynomial of $L$ in terms of $x_1, x_2$ and $x_3$.

Notice that  this relation can be rewritten as $(L-M)^2(L-2M)=0$  (this follows immediately from factorisation of the characteristic polynomial on the initial line $x(\tau)$).   But this factorisation immediately allows us to find the eigenvalues of $L$ by taking the roots of the polynomial
$$
(\lambda - x_1 \lambda^2 - x_2\lambda - x_3)^2 (\lambda - 2x_1 \lambda^2 - 2x_2\lambda - 2x_3)=0
$$
Since we are working in a neighbourhood of the origin, we are interested in specific roots, namely those which, on the initial curve, coincide with  the above prescribed roots, that is, 
$$
\lambda_1(0,0,x_3) = \lambda_2(0,0,x_3) = x_3, \quad \lambda_3(0,0,x_3) = 2x_3. 
$$

In this particular case we just need to choose the right root  (one of the two) of the corresponding quadratic equation. Namely,
$$
\lambda - x_1 \lambda^2 - x_2\lambda - x_3=0   \quad \Rightarrow \quad  \lambda = \frac{2x_3}{(1-x_2)+\sqrt{(1-x_2)^2 - 4x_1x_3}}
$$
$$
\lambda - 2x_1 \lambda^2 - 2x_2\lambda - 2x_3=0   \quad \Rightarrow \quad  \lambda = \frac{4x_3}{(1-2x_2)+\sqrt{(1-2x_2)^2 - 16x_1x_3}}
$$

The root of the first equation is an eigenvalue of $L$ of multiplicity 2, whereas the root of the second equation is an eigenvalue of multiplicity one. As a result we have found explicit expressions for the eigenvalues of the Nijenhuis operator $L$ in coordinates $x_1, x_2, x_3$:
\begin{equation}
\label{eq:2}
\lambda_1=\lambda_2 = \frac{2x_3}{(1-x_2)+\sqrt{(1-x_2)^2 - 4x_1x_3}}, \quad \lambda_3 = \frac{4x_3}{(1-2x_2)+\sqrt{(1-2x_2)^2 - 16x_1x_3}}
\end{equation}

The final conclusion is that the operator 
$$
\Lfirst = \begin{pmatrix}
f_1(x) & 1 & 0\\
f_2(x) & 0 & 1 \\
f_3(x) & 0 & 0
\end{pmatrix}  \quad\mbox{with } \ 
\begin{array}{rl}
f_1 &= \lambda_1 + \lambda_2 +\lambda_3, \\
f_2 &= - \lambda_1\lambda_2 - \lambda_2\lambda_3 - \lambda_3\lambda_1, \\
f_3 &= \lambda_1\lambda_2\lambda_3,
\end{array}
$$
where $\lambda_i$ are defined by \eqref{eq:2} is a Nijenhuis operator in first companion form.  This is an example of a Nijenhuis perturbation of the nilpotent $3\times 3$ Jordan block $J_0$ under which $J_0$ splits into two Jordan blocks of size $2$ and $1$ with non-constant eigenvalues.  
}\end{Ex}

\section{Local classification of $\gl$-regular Nijenhuis operators in dimension two and its applications}\label{sect:dim2}

The goal of this section is to describe local normal forms for $\gl$-regular  Nijenhuis operators at singular points  in dimension 2. 
However, for the sake of completeness we first recall the list of (algebraically) generic types of such operators along with their {\it local} canonical forms:

\begin{itemize}
\item Two distinct real eigenvalues: 
$
L=\begin{pmatrix}
f(x) & 0 \\ 0 & g(y)
\end{pmatrix}
$,   where $f(x)$ and $g(y)$ are smooth functions such that $f(x) \ne g(y)$ for all $(x,y)$.  In the real analytic case, $f(x)$ is either constant or can be reduced, by an appropriate local change of coordinates,  to  $f(x)=f_0 \pm x^{2m}$ or $f(x)=f_0 + x^{2m-1}$, $m\in\mathbb N$, and similarly for $g(y)$.
\item Two complex conjugate eigenvalues:
$
L=\begin{pmatrix}
f(x,y) & -g(x,y) \\
g(x,y) & f(x,y)
\end{pmatrix}
$,  where $h=f + \mathrm{i}\, g$ is a holomorphic function of the complex variable $z=x+
\mathrm{i}\,y$, $g(x,y)\ne 0$ for all $(x,y)$.  This function $h(z)$ is either constant or can be reduced, by an appropriate local change of coordinates, to  $h(z) = h_0 + z^m$, $m\in\mathbb N$.
\item  Jordan block:
$
L=\begin{pmatrix}
f(y) & 1 \\ 0 & f(y)
\end{pmatrix}
$, where $f(y)$ is a smooth function. As above, in the real analytic case, $f(y)$  is either constant or can be reduced, by an appropriate local change of coordinates,  to  $f(x)=f_0 \pm x^{2m}$ or $f(x)=f_0 + x^{2m-1}$, $m\in\mathbb N$.
\end{itemize} 

This classification is easy and well known.  A non-trivial problem is to describe local behaviour of $L$ near a singular point $\mathsf p$ at which the algebraic type of $L$ changes.  In dimension 2 under the $\gl$-regularity assumption,  there is only one possibility for $L(\mathsf p)$, namely,  this operator (after an appropriate chage of coordinates) is a Jordan block:
$$
L(\mathsf p) = \lambda \operatorname{Id} + J_0, \quad \mbox{where } J_0 = \begin{pmatrix} 0 & 1 \\ 0 & 0 \end{pmatrix}, \ \lambda=\mathrm{const}\in\R.
$$  

Since $L - \lambda \operatorname{Id}$ is still a Nijenhuis operator,  we will assume w.l.o.g. that $L(\mathsf p) = J_0$ and our problem reduces to classification of Nijenhuis perturbations of the nilpotent Jordan block $J_0$.  Below we will describe all possible normal forms for such perturbations, i.e., for Nijenhuis operators $L$ such that $L(\mathsf p)=J_0$.  To our great surprise, they are all polynomial. Before stating our classification result, we notice that there are two essentially different cases depending on the coefficients of the characteristic polynomial 
$$
\chi_L(\lambda) = \det(\lambda{\cdot}\operatorname{Id} - L) = \lambda^2 - v \lambda - u, \quad  v=\tr L, \ u = -\det L.
$$ 
In the real analytic case, there are two possibilities:  either $\ddd v \wedge \ddd u \equiv 0$ or $\ddd v \wedge \ddd u \ne 0$ on an open everywhere dense subset. In the latter case, the operator $L$ can be completely reconstructed from $v$ and $u$ and the relation  (see \cite[Corollary 2.2]{Nijenhuis1}):
\begin{equation}
\label{eq:bols01}
L = \begin{pmatrix} v_x & v_y \\ u_x  & u_y  \end{pmatrix}^{-1} 
\begin{pmatrix}
v & 1 \\ u & 0 
\end{pmatrix} \begin{pmatrix} v_x & v_y \\ u_x  & u_y  \end{pmatrix}, \qquad v=\tr L, \ u=-\det L.
\end{equation}
At those points where the Jacobi matrix is not invertible, we define $L$ by continuity.  In other words, in the above formula we should automatically observe ``cancellation of the denominator''  $v_xu_y - v_y u_x$  involved in the formula of the inverse matrix.
For this reason in  Theorem \ref{thm:dim2} below,   when appropriate, instead of the matrix of $L$ we will give formulas for $v(x,y)$ and $u(x,y)$  as they are much simpler and more intuitive.  The reader may easily ``reconstruct'' $L$ from \eqref{eq:bols01} and, in particular, see the above mentioned cancellation.  

If $\ddd v \wedge \ddd u \equiv 0$, then \eqref{eq:bols01} makes no sense, but we may still use another more general relation
(see \cite[Proposition 2.2]{Nijenhuis1}):
\begin{equation}
\label{eq:bols03}
\begin{pmatrix} v_x & v_y \\ u_x  & u_y  \end{pmatrix} \begin{pmatrix} l^1_1 &  l^1_2 \\  l^2_1 & l^2_2 \end{pmatrix} = 
\begin{pmatrix}
v & 1 \\ u & 0 
\end{pmatrix} \begin{pmatrix} v_x & v_y \\ u_x  & u_y  \end{pmatrix}, \qquad v=\tr L, \ u=-\det L, \ L=\begin{pmatrix} l^1_1 &  l^1_2 \\  l^2_1 & l^2_2 \end{pmatrix}.
\end{equation}

We will assume that $L$ is defined in a neighbourhood of the origin $\mathsf p=(0,0)\in\R^2(x,y)$ and coordinate transformations always leave the origin fixed. 
The theorem below provides the complete list of normal forms for $L$ which are divided into several series.

\begin{Theorem}\label{thm:dim2}
Let $L$ be a Nijenhuis operator such that $L(\mathsf p) = \begin{pmatrix} 0 & 1 \\ 0 & 0 \end{pmatrix}$. Then in suitable local coordinates $(x,y)$,  this operator takes one of the following forms:

\begin{enumerate}

\item    Series  $L, M$ and $N$ {\rm(}for $k \geq 1, \epsilon = \pm 1${\rm)}: 
    \begin{equation}
\begin{aligned}
L_{\mathrm{nil}} = \left(\begin{array}{cc} 0 & 1 \\ 0 & 0 \end{array} \right), \quad 
L_{\mathrm{nd}} & = \left(\begin{array}{cc} x & 1 \\ y & 0 \end{array} \right), \quad 
M_{2k - 1} =  \left(\begin{array}{cc} 0 & 1 \\ 0 & y^{2k - 1}\!\! \end{array} \right), \quad
M^{\epsilon}_{2k} =  \left(\begin{array}{cc} 0 & 1 \\ 0 & \epsilon y^{2k} \! \end{array} \right), \\
N_{2k - 1} & =  \left(\begin{array}{cc} \!\! y^{2k - 1} & 1 \\ 0 & y^{2k - 1} \!\! \end{array} \right), \quad 
N^{\epsilon}_{2k} =  \left(\begin{array}{cc} \!\! \epsilon y^{2k} & 1 \\ 0 & \epsilon y^{2k} \!\! \end{array} \right)
\end{aligned}
\end{equation}

\item Series $O^{d,\epsilon}_{k, \mathrm{c}}$,  $k\ge 1$,  $d\ge 2k+1$, $\epsilon=\pm 1$, $\mathrm c=(c_0,\dots, c_{k-1})\in\R^k$ and we set $\epsilon=1$, if $d=2m+1$ is odd.  

The operator $L$ is defined by \eqref{eq:bols01} with $v=\tr L$ and $u=-\det L$ given by
$$
 v =  \alpha xy^{2k-1} + y^k \bigl(c_{k-1} y^{k-1} + \dots + c_1 y + c_0\bigr), \quad u= \epsilon \, y^d, \quad \alpha=k c^2_0 \left(1  - \frac{k}{d}\right)\ne 0.
 $$
 
\item Series $P^{k,\epsilon}_{s, \mathrm{c}}$,  $k\ge 1$,  $s\ge 2k$, $\epsilon=\pm 1$, $\mathrm c=(c_0,\dots, c_{k-1})\in\R^k$.

The operator $L$ is defined by \eqref{eq:bols01} with $v=\tr L$ and $u=-\det L$ given by
$$
v = \alpha x y^s +  y^{s-k+1} \bigl(c_{k-1} y^{k-1} + \dots + c_1 y + c_0\bigr) + 2\epsilon\, y^k, \quad u=  - y^{2k}, \quad  \alpha = 2\epsilon\, k c_0 \ne 0.
$$

\item Series $S^{2k, \epsilon}_{\mathrm{c}}$ and $S^{2k+1}_{\mathrm{c}}$,   $k\ge 1$,  $\mathrm c=(c_0,\dots, c_{k-1})\in\R^k$.

The operator $L$ is defined by \eqref{eq:bols01} with $v=\tr L$ and $u=-\det L$ given respectively by
$$
\begin{aligned}
v&= \alpha xy^{2k-1} + y^k \bigl(c_{k-1} y^{k-1} + \dots + c_1 y + c_0\bigr), \quad u = \epsilon \, y^{2k}, \quad \alpha=\frac{k}{2} \left(  c_0^2 + 4\epsilon \right)\ne 0, \\
 v&= \alpha xy^{2k} + y^{k+1} \bigl(c_{k-1} y^{k-1} + \dots + c_1 y + c_0\bigr), \quad  u=y^{2k+1}, \quad \alpha = 2k+1. 
\end{aligned}
$$ 
\end{enumerate}
\end{Theorem}


\begin{proof}
The idea of the proof is natural:  since $L$ is  basically defined by its trace and determinant,  we will be looking for local coordinates $x,y$ in which $v=\tr L$ and $u=-\det L$ have their ``simplest'' possible form.  We start with two technical lemmas.

\begin{Lemma}\label{lem0}  Under assumptions of Theorem \ref{thm:dim2}, there exist local coordinates $(x,y)$ such that for $u=-\det L$ one of the following holds:
$$\ 
(\mathrm{i})  \  u \equiv 0, \quad   (\mathrm{ii}) \ u = \pm \, y^{2k},  \quad (\mathrm{iii}) \  u=y^{2k-1},  \qquad k\in\mathbb N.
$$
\end{Lemma}
\begin{proof}   In companion coordinates (see \eqref{first_set}),  the function $u=-\det L$ satisfies the equation 
\begin{equation}
\label{eq:bols02}
u_x =  g(x,y) u,
\end{equation}  
where $g(x,y)= \partial_y \tr L$. Hence $u = f(y) \exp\left(\int_0^x g(t,y)dt\right)$ for some real analytic function $f(y)$.  If $f(y)\equiv 0$, we have Case (i). Otherwise, writing $f$ in the form $f(y) = \epsilon y^m h(y)$ with $\epsilon =\pm\, 1$, $h(0) >0$, $m\in\mathbb N$, we get for $m=2k$ and $m=2k-1$ respectively:
$$
u =  \pm \left(  y \sqrt[2k]{h(y)\exp\left(\int_0^x g(t,y)dt\right) }  \right)^{2k} \quad \mbox{or} \quad
u =  \left(  y \sqrt[2k-1]{\pm h(y)\exp\left(\int_0^x g(t,y)dt\right) }  \right)^{2k-1} .
$$
Letting $y_{\mathrm{new}}$ be the expression in brackets gives $u=\pm\, y_{\mathrm{new}}^{2k}$ or $u=y_{\mathrm{new}}^{2k-1}$, as required. \end{proof}

This lemma brings $\det L$ to its simplest canonical form.  After this we keep the $y$-coordinate fixed and simplify $v=\tr L$ by changing the $x$-coordinate only. 
    
\begin{Lemma}\label{lem1'}
There exists a coordinate change of the form $(x_{\mathrm{old}},y) \mapsto (x, y)$ such that the $l^1_2$-component of $L$ in new coordinates equals identically $1$.
\end{Lemma}

\begin{proof} Setting $x_{\mathrm{old}} = g(x, y)$ and applying the standard transformation rule for components of an operator, we observe that the required condition is 
$$
\frac{ l^1_2 (g, y) + g_y l^1_1(g, y) - g_y l^2_2(g, y)  - g_y^2 l^2_1(g, y) }{g_x} = 1,
$$
where $l^i_j$ are the components of $L$ in the old coordinate system.
Writing this relation in the form $g_x = F (g_y, g, y)$, we can locally solve it by Cauchy-Kovalevskaya theorem.  It is important that $L$ is 
$\gl$-regular, this allows us to choose initial conditions in such a way that $g_x(0,0)\ne 0$ so that the coordinate transformation is invertible. \end{proof}

Now let us discuss all the cases one by one.
First, assume that $u \equiv 0$ and $v \equiv 0$, then $L$ is a nilpotent Jordan block and its companion form coincides with $L_{\mathrm{nil}}$. 

Next suppose $u \equiv 0$, while $v$ is not. In companion coordinates (see \eqref{first_set}), $v$ satisfies the Hopf equation $v v_y - v_x = 0$. This equation can be rewritten as 
\begin{equation}
v_x = g(x.y) v \quad \mbox{with } g = v_x,    
\end{equation}
which is similar to the above equation \eqref{eq:bols02} for $u$.  Just in the same way as in Lemma \ref{lem0}, we find a coordinate system in which $v = y^{2k - 1}$ or $v=\epsilon y^{2k}$ for $k \geq 1, \epsilon = \pm 1$. By Lemma \ref{lem1'}, we may also assume that $l^1_2 = 1$.   Now $L=\Bigl(l^i_j\Bigr)$ can be reconstructed from relation   \eqref{eq:bols03}. This yields series $M_{2k - 1}$ and $M^{\epsilon}_{2k}$ for different $v$ respectively.

Now let $u \not \equiv 0$, but $\ddd v \wedge \ddd u \equiv 0$.  Combining Lemmas \ref{lem0} and \ref{lem1'}, we may assume that  $u = y^{2m - 1}$ or $u=\pm \, y^{2m}$ for $m \geq 1$, and  $l^1_2 = 1$. Since $\ddd v \wedge \ddd u \equiv 0$, we also know that $v_x \equiv 0$. Relation \eqref{eq:bols03} implies that $l^2_1 = 0$ and we come to the operator of the form
$$
L = \begin{pmatrix}  f(y) & 1 \\ 0 & g(y)   \end{pmatrix} \quad \mbox{with }  v = f+g \ \mbox{and} \ u = -fg.
$$
It is straightforward to check that Nijenhuis condition in this case reads $f'_y(f-g)=0$. In our case $f$ cannot be constant as in this case, since $L$ is nilpotent at the origin, we would necessarily have $f\equiv 0$,  which contradicts our assumption that $u=-fg\not\equiv 0$.     Therefore, we conclude that $f-g = 0$,  meaning that $L$ is a Jordan block at each point.
 This yields series $N_{2k - 1}$ and $N^{\epsilon}_{2k}$.

If $\ddd v \wedge \ddd u \ne 0$ at the point $\mathsf p$, then $L$ is differentially non-degenerate and its normal form is $L_{\mathrm{nd}}$  \cite[Theorem 4.4]{Nijenhuis1}.  Notice that in terms of Lemma \ref{lem0},  the non-degeneracy condition corresponds exactly to the case $u=y$ and below we exclude this case.

Finally we consider the most interesting case when $\ddd v \wedge \ddd u \not \equiv 0$  (but $\ddd v \wedge \ddd u = 0$ at $\mathsf p$).  As previously, we assume that $u = y^{2m + 1}$ or $u=\epsilon y^{2m}$ for $m \geq 1, \epsilon = \pm 1$ and $l^1_2 = 1$.   Computing $l^1_2$ from matrix relation  \eqref{eq:bols01} yields the following equation on $v$:
\begin{equation}\label{rff}
v_x = v v_y - \frac{1}{d}y (v_y)^2 + u_y, 
\end{equation}
where $d = 2m + 1$ or $2m$.   This equation implies the following

\begin{Lemma}\label{lem2}
The function $v(x,y)$ can be written as $v = v_0(y) + y^s (\alpha x + F)$, where $\alpha \neq 0$, $s \geq 1$ and $F(x, y)$ is a real analytic function with no constant or linear part.
\end{Lemma}

\begin{proof}
Let $v(x,y) = v_0(y) + v_1(y) x + v_2(y) x^2 + \dots$ be a solution of \eqref{rff}. 
Differentiating \eqref{rff} w.r.t. $x$ we get
$$
v_{xx} = v_x v_y + v v_{xy} - \frac{2}{d}y v_y v_{xy}.
$$
Note that $v = v_0(y)$ satisfies this equation for initial conditions $v(0, y) = v_0(y), v_x(0, y) = v_1(y) \equiv 0$. By Cauchy-Kovalevskaya theorem this solution is unique.  Hence, if $v_1 \equiv 0$, then $\ddd v \wedge \ddd u \equiv 0$ which is wrong. Thus, in our case $v_1(y) = y^s r_1(y)$, where $r_1(0) = \alpha \neq 0$.

Assume that $s = 0$. This means that $\ddd v$ and $\ddd y$ are linearly independent. We can introduce $x_{\mathrm{new}}=v=\tr L$, leaving $y$ the same. In these new coordinates,  relation \eqref{eq:bols01} gives
$$
L = \begin{pmatrix}  1 & 0 \\  0 & u_y^{-1} \end{pmatrix} \begin{pmatrix} x_{\mathrm{new}} & 1 \\  u(y) & 0    \end{pmatrix}\begin{pmatrix}  1 & 0 \\  0 & u_y \end{pmatrix} =
\begin{pmatrix}    x_{\mathrm{new}}  & u_y \\ u u_y^{-1} & 0  \end{pmatrix}.
$$
It is easy to see that $L$ at the origin $\mathsf p$ is similar to the nilpotent Jordan block only for $u= y$. But in this case we get $L=L_{\mathrm{nd}}$ falling into the previous case.

Thus,  we have $s \geq 1$.  Equating the coefficients of $x^i$ it the both sides of \eqref{rff} yields
\begin{equation}\label{rush}
\begin{aligned}
(i + 1) v_{i + 1}(y) & =  v_0(y)v_i'(y) + v'_0(y)v_i(y) - \frac{2}{d} yv'_0(y)v'_i(y) + \\
& + \sum \limits_{j = 1}^{i - 1} \big(v_j(y)v'_{i - j}(y) - \frac{1}{d}y v'_j(y)v'_{i - j}(y)\big).
\end{aligned}
\end{equation}
If $v_1, \dots, v_i$ are divisible by $y^s$, then $v'_1, \dots, v'_i$ are divisible by $y^{s - 1}$. As $v(0, 0) = 0$, then $v_0$ is divisible by $y$. By formula \eqref{rush} the coefficient $v_{i + 1}$ is divisible by $y^s$. Thus, by induction all the coefficients $v_1, v_2, \dots$ are divisible by $y^s$ and one writes $v = v_0 + y^s(x r_1(y) + \dots) = v_0 + y^s(\alpha x + F)$, where $F$ is analytic and has no constant or linear parts. Lemma is proved.
\end{proof}

Using Lemma \ref{lem2}, we introduce new coordinates $x_{\mathrm{new}} = x + \frac{1}{\alpha}F + \tilde{v_0}$, $y_{\mathrm{new}} = y$, where $\tilde{v_0}$ contains all the terms of $v_0$ of order $\ge s + 1$. In this new coordinate system  (for which we continue using old notation  $x$ and $y$) we have:
\begin{equation}
\label{eq:bols04}
v = p_s(y) +\alpha x y^s, \quad  u = y^{2m + 1} \ \mbox{or} \ u=\epsilon y^{2m},
\end{equation}
where $\alpha \neq 0$, $m, s \geq 1$ and $p_s$ is polynomial of degree at most $s$. 

This coordinate system is optimal in the sense that $v=\tr L$ and $u=-\det L$ cannot be simplified further.  The last step is to distinguish those pairs of functions $v(x,y)$ and $u(x,y)$  from family \eqref{eq:bols04} that indeed {\it generate}  analytic perturbations of the nilpotent Jordan block $J$ via relation \eqref{eq:bols01}.  The point is that \eqref{eq:bols01} will generate a Nijenhuis operator $L$ for any $v$ and $u$, but we need only those of them which have no singularity at $\mathsf p=(0,0)$ and, moreover, such that $L(\mathsf p)$ is similar to $J_0$.  

Straightforward reconstruction of $L$,  from \eqref{eq:bols01}  with $v$ and $u$ given by \eqref{eq:bols04},  shows that all the components of $L$ are non-singular and vanish at the origin except for $l^1_2$:
$$
L = \begin{pmatrix}
v - \frac{yv_y}{d} &   \frac{v v_y  -  \frac{1}{d} y v_y^2 + u'}{v_x}    \\  &  \\
\frac{y v_x}{d} &  \frac{y v_y}{d}
\end{pmatrix}
$$
where $d=2m+1$ or $d=2m$ (power of $y$ in the formula for $u$).   

The ``troublesome'' component, in more detail, reads:
$$
l^1_2 = s \alpha x^2 y^{s-1} \left(1-\frac{s}{d}\right) + x\left( \frac{1}{y} p_s + \left( 1-\frac{2}{d} \right) p_s'   \right) + \frac{p_s p'_s - \frac{1}{d}y(p'_s)^2 + u'}{\alpha y^s}. 
$$ 
Notice that $p_s(0)=0$  and therefore  $\frac{1}{y} p_s$ is analytic.  Hence, we only need to analyse the fraction
\begin{equation}\label{algebraic}
    \frac{p_s p'_s - \frac{1}{d}y(p'_s)^2 + u'}{\alpha y^s}   
\end{equation}
This fraction must define an analytic function having value $1$ at the origin (in order for $L(\mathsf p)$ to be the standard nilpotent Jordan block). Thus, we need to solve a purely algebraic problem: find all polynomials $p_s$, for which the denominator of \eqref{algebraic} is divisible by $\alpha y^s$ so that this fraction is, in fact,  a polynomial with free term equal to $1$.  We rewrite \eqref{algebraic} as
\begin{equation}\label{alg}
    p_s p'_s - \frac{1}{d}y(p'_s)^2 = - u' + \alpha y^s + \alpha_1 y^{s + 1} + \dots + \alpha_{s - 1}y^{2s - 1},
\end{equation}
where $\alpha \neq 0$ and $\alpha_i$ are, in general, arbitrary.

Let $p_s$ starts with a term of order $k\ge 1$, that is, $p_s = y^k\bigl(c_0 + c_1 y + \dots + c_{s-k} y^{s - k}\bigr)$. Then the smallest degree term in the l.h.s. of \eqref{alg}  is $k c_0^2 \left(1 - \frac{k}{d}\right)y^{2k - 1}$. On the other hand, the term of the smallest degree in the r.h.s. is either $u'=\pm d y^{d-1}$ or  $\alpha y^s$ (or both of them).  

First, assume $s < d - 1$. Then we get $2k - 1 = s$ and furthermore $p_{2k - 1} = c_0 y^{k} + \dots + c_{k -1} y^{2k - 1}$, where $c_1, \dots, c_{k-1}$ are arbitrary and $c_0 \neq 0$. We also have  $\alpha = k c_0^2 \left(1 - \frac{k}{d}\right)$ obtaining, as a result, the series $O^{d,\epsilon}_{k,\mathrm{c}}$. 

Next, assume $d - 1 < s$. Then we get $d - 1 = 2k - 1$ and, thus, $u = \epsilon y^{2k}$. Equating the coefficients of $y^{2k - 1}$ on both sides of \eqref{alg} we get $\frac{k}{2} c^2_0 = - 2k \epsilon$ and, thus, $u = - y^{2k}$ and $c_0=\pm 2$. We write  $p_s = \pm 2 y^k + c_1 y^{k + 1} + \dots + c_{s - k} y^s$ and substitute it into  \eqref{alg}. Equating the coefficients of $y^{2k},\dots, y^{s-1}$ in the l.h.s. of \eqref{alg} to zero we get, step by step, that  $c_1=c_2=\dots =c_{s-2k}=0$. Hence,  
re-denoting $c_{s-2k+j} \mapsto c_{j-1}$ for $j=1,\dots, k$ we have: 
$$
p_s = y^{s-k+1}\bigl(c_{k-1} y^{k-1} + \dots + c_1 y + c_0\bigr) \pm 2y^k,
$$
and equating the coefficients of $y^s$ in both sides of \eqref{alg},  we  obtain  $\alpha =  \pm 2 k c_0\ne 0$.   This yields series $P^{k,\epsilon}_{s, c}$. 

Finally, consider $d - 1 = s$. We have two possibilities. First, assume that $d = 2m$, i.e., $u = \epsilon y^{2m}$. We get that $2k - 1 = 2m - 1$, $k = m$ and $v = \alpha xy^{2m - 1} + c_0 y^m + \dots + c_{m-1} y^{2m - 1}$ with $\alpha = \frac{m}{2} (c_0^2 + 4\epsilon) \neq 0$. Now assume that $d = 2m + 1$, i.e., $u = y^{2m + 1}$. This yields $v = \alpha xy^{2m} + c_0 y^{m + 1} + \dots + c_{m-1} y^{2m}$ and $\alpha = 2m + 1$. This yields $S^{2m,\epsilon}_{\mathrm{c}}$ and $S^{2m+1}_{\mathrm{c}}$ respectively (in the statement of the theorem we replace $m$ by $k$). \end{proof}

\begin{Remark}{\rm   For the series $O$, $P$ and $S$ the canonical coordinate system is essentially unique  (in some cases on can simultaneously change the sign of $x$ and $y$).  Indeed,   these coordinates are those in which $u=-\det L$ and $v=\tr L$ are given by \eqref{eq:bols04}. The integer parameters $m$ and $s$ involved in \eqref{eq:bols04} are uniquely defined for given $u$ and $v$.  Hence,  $y$ can be reconstructed from $u$ (sometimes up to sign), and $x$ is determined, up to a constant factor, by the condition that $v(0,y)$ is a polynomial of degree $\leq s$. Finally, the rescaling  of $x$ is chosen in such a way that at the origin we have $L(0,0)=J_0$.

This implies that Nijenhuis operators from different series (or from the same series but with different parameters) are not equivalent to each other.  The only exception is related to the above mentioned ``canonical'' transformation $(x,y) \mapsto (-x,-y)$ that changes the parameter $\mathrm c \in \R^k$, but this change is easy to control.  
}\end{Remark}

We now apply the local classification of $\gl$-regular Nijenhuis operators to study the existence  (and examples) of such operators on closed two-dimensional surfaces.

Let $(\mathsf{M}^2, L)$ be a $\gl$-regular Nijenhuis manifold of dimension $2$  (recall that we always assume them to be real analytic).      
Consider the set $\mathsf{Sing}$ of singular points of $L$ where the algebraic type of $L$ changes.   In our case,  this means that the eigenvalues of $L$ collide, i.e.
$$
\mathsf{Sing}=\{  \mathsf p\in \mathsf{M}^2~|~  v^2 + 4u =0\}, \quad \mbox{where }  v=\tr L, \ u=-\det L,
$$
unless $v^2 +4u \equiv 0$ on $\mathsf{M}^2$ meaning that $L$ is similar to a Jordan block at each point.

From Theorem \ref{thm:dim2} we immediately obtain a {\it local} description of $\mathsf{Sing}$ in canonical coordinates $x,y$: 

\begin{itemize}
\item for $L_{\mathrm{nil}}$, $N_{2k-1}$ and $N_{2k}^\epsilon$, the singular set is empty; 
\item for $L_{\mathrm{nd}}$ the singular set is $\mathsf{Sing}=\{ x^2 +4y=0\}$;
\item for all the other series $M$, $O$, $P$ and $S$:  \quad
$\mathsf{Sing}_{loc}=\{ y=0\}$.
\end{itemize}

Thus, locally $\mathsf{Sing}$ is a smooth curve. Since $\mathsf{Sing}\subset \mathsf{M}^2$ is closed, we may think of it as a submanifold consisting, perhaps, of several connected components:   
$$
\mathsf{Sing} = \mathsf{S}_1 \cup \dots \cup \mathsf{S}_\ell.
$$
If $\mathsf{M}$ is compact, then each of them is an embedded circle.  Next we can easily observe that all points from  $\mathsf S_i$ relate to the same series  (different components may, of course, relate to different series).   However, the parameters of the series may change.   This happens for series $O$, $P$ and $S$. Indeed,  moving  along $\mathsf{Sing}_{loc}=\{ y=0\}$ leads to the shift $x_{\mathrm{new}}=x - x_0$ resulting in the following modification for $v=\tr L$   (whereas $\det L$ remains unchanged):
$$
v = \alpha x y^s + c_{k-1} y^s +  \dots =\alpha  (x_{\mathrm{new}} + x_0)  y^s + c_{k-1} y^s  + \dots = \alpha  x_{\mathrm{new}}   y^s + (c_{k-1}+\alpha x_0) y^s  + \dots
$$
In other words, all parameters remain fixed except for $c_{k-1}$  which undergoes the shift $c_{k-1} \mapsto c_{k-1}+\alpha x_0$.  Notice that if we move along $\mathsf S_i$ in a certain direction, then $c_{k-1}$ is either strictly increasing or strictly decreasing.  This leads us to the following conclusion.

\begin{Proposition}\label{prop4.1}
Singular points from the series $O$, $P$ and $S$ may not occur on closed $\gl$-regular Nijenhuis 2-manifolds.
\end{Proposition}     

According to  \cite[Proposition]{Nijenhuis1},  the same conclusion holds for differentially non-degenerate singular points (series $L_{\mathrm{nd}}$) and therefore we obtain 

\begin{Proposition}\label{prop4.2}
Let  $(\mathsf{M}^2, L)$ be a closed $\gl$-regular Nijenhuis 2-manifold.  Then 
\begin{itemize}
\item either $\mathsf{Sing}$ is empty  (i.e. all points of $\mathsf{M}^2$ are of the same algebraic type), 
\item or   each $\mathsf p\in \mathsf{Sing}$  belongs to the series $M$ and then  automatically one of the eigenvalues of $L$ is constant on $\mathsf{M}^2$.
\end{itemize}
\end{Proposition}

We are now ready to prove our final result.

\begin{proof}[Proof of Theorem \ref{main:3}]  Consider the two options from Proposition \ref{prop4.2}. First assume that $\mathsf{Sing} =\emptyset$.   Then  $L$ belongs to one of three generic types listed in the beginning of this Section:
\begin{enumerate}
\item[(i)] either $L$ has two distinct real eigenvalues  at each point of $\mathsf M^2$;
\item[(ii)] or $L$ has two complex conjugate eigenvalues  at each point of $\mathsf M^2$;
\item[(iii)] or $L$ is similar to a Jordan block at each point of $\mathsf M^2$.
\end{enumerate}

In Case (i),  at each point  $\mathsf p\in\mathsf M^2$,  we have an eigenbasis basis  $e_1, e_2\in T_{\mathsf p} \mathsf M^2$ where $e_1$ corresponds to the maximal eigenvalue at a given point.  If we fix some Riemannian metric on $\mathsf M^2$, we may assume that $e_i$ are normalised so that $|e_i|=1$.  Since such $e_i$ are defined up to $\pm$, we have 4 different bases at each point.  A priori it is not clear whether or not we can chose a smooth ``moving frame'' field on the whole manifold, but this can obviously be done on a finite sheeted covering $\widetilde{\mathsf M}^2$ of $\mathsf M$  (number of sheets is at most four).  This implies that $\widetilde{\mathsf M}^2$ is parallelisable and hence is a torus.  Therefore, $\mathsf M^2$ is either a torus or Klein bottle and we obtain Case 2 of Theorem \ref{main:3}.

In Case (ii), according to \cite[Theorem 6.1]{Nijenhuis1} the complex eigenvalues $\lambda,\bar\lambda$ of the Nijenhuis operator $L$ are constant and we 
obtain Case 1 from Theorem \ref{main:3}.

In Case (iii), at each point $\mathsf p\in\mathsf M^2$ we have a non-zero eigenvector $e \in T_{\mathsf p} \mathsf M^2$ and the same argument as above shows that on $\mathsf M^2$ or on its two sheeted covering one can define a smooth vector field with no singular points. Hence $\mathsf M^2$ is either a torus or Klein bottle.   However,  in this case we have one additional property that the automorphism group of a Jordan block consists of orientation preserving transformations,  which allows us to define orientation on $\mathsf M^2$. Hence, the Klein bottle is forbidden and we are lead to Case 3 of Theorem \ref{main:3}.

Thus, the condition $\mathsf{Sing}=\emptyset$ necessarily implies one of the first three cases of Theorem \ref{main:3}.

Finally, we consider the second option from Proposition \ref{prop4.2}.   This option implies that one of the eigenvalues of $L$ is constant allowing us to consider a non-zero eigenvector related to this eigenvalue at each point  and, in the same way as above, to construct a smooth vector field with no zeros either on $\mathsf M^2$ or on its two-sheeted covering.  This implies that $\mathsf M^2$ is either a torus or Klein bottle and we obtain Case 4 of Theorem \ref{main:3}.

Thus, the list of possibilities presented in Theorem \ref{main:3} is complete.  \end{proof}

We conclude this section with examples of Nijenhuis operators listed in Theorem \ref{main:3}.

\begin{Ex}
\label{ex:torus1}
{\rm

Let $\mathsf T^2$ be a torus with standard angle coordinates $\phi_1$ and $\phi_2$ defined modulo $2\pi$.  For an operator $L$ with two distinct eigenvalues at each point $(\phi_1, \phi_2)$, we can distinguish three essentially different possibilities. 

\begin{itemize}  

\item Two constant eigenvalues $\lambda_1$ and $\lambda_2$. Let $\xi$ and $\eta$ be two vector fields on $\mathsf{T}^2$ that are linearly independent at each point (NB: there are many non-equivalent examples of such vector fields), then we define $L$  by setting
\begin{equation}
\label{eq:bols12}
L(\xi)=\lambda_1 \xi\quad \mbox{and} \quad L(\eta)=\lambda_2 \eta.
\end{equation}

\item One constant eigenvalue  (w.l.o.g.  $\lambda_1 =0$), the other $\lambda_2$ is not.   In coordinates $(\phi_1,\phi_2)$ we define $L$ as 
\begin{equation}
\label{eq:bols13}
L = \begin{pmatrix}  0 & g(\phi_1,\phi_2) \\ 0 & f(\phi_2) \end{pmatrix}
\end{equation}   
with $f(\phi_2) >0$  or $f(\phi_2) <0$. Here $\xi=\left(1, -\frac{g(\phi_1,\phi_2)}{f(\phi_2)}\right)$ is an eigenvector field related to the non-constant eigenvalue $\lambda_2 = f(\phi_2)$. 

\item Two non-constant eigenvalues $\lambda_1$ and $\lambda_2$.   An obvious example is 
\begin{equation}
\label{eq:bols14}
L = \begin{pmatrix}  f(\phi_1) & 0 \\ 0 & g(\phi_2) \end{pmatrix}, \qquad  f(\phi_1) < c < g(\phi_2).
\end{equation} 
This example can be modified by taking a finite-sheeted covering over this ``standard'' torus. On the covering torus,  the above {\it global diagonalisation} of $L$ is not always possible.
\end{itemize}
}\end{Ex}

\begin{Ex}\label{ex:Klein1}
{\rm  Each of the above examples \eqref{eq:bols12}, \eqref{eq:bols13} and \eqref{eq:bols14} can be naturally ``transferred'' to the Klein bottle $\mathsf K^2$ that can be thought of as the quotient of $\mathsf T^2$ with respect to the involution $\sigma: \mathsf T^2 \to \mathsf T^2$ given by $(\phi_1,\phi_2) \overset{\sigma}{\mapsto} (-\phi_1, \phi_2+\pi)$.  We only need to make sure that $L$ is invariant with respect to $\sigma$.    Namely, in the above three cases from Example \ref{ex:torus1}  we assume in addition that
\begin{itemize}
\item  $\xi$ is $\sigma$-invariant, whereas $\eta$ changes the direction under the action of $\sigma$, i.e., $\ddd\sigma(\xi) = \xi$ and $\ddd\sigma(\eta) = -\eta$, 

\item $f(\phi_2)$ is $\pi$-periodic,  $g(\phi_1,\phi_2)$ is even w.r.t. $\phi_1$,

\item $g(\phi_2)$ is $\pi$-periodic and $f(\phi_1)$ is even.
\end{itemize}
If these conditions are fulfilled, then the operators $L$ given by \eqref{eq:bols12}, \eqref{eq:bols13} and \eqref{eq:bols14} naturally descend to the quotient $\mathsf{K}^2 = \mathsf{T}^2/\sigma$.
}\end{Ex}

The next is an example of a Nijenhuis operator on $\mathsf T^2$ of Jordan block type (see Case 3 in Theorem \ref{main:3}.

\begin{Ex}{\rm
Assume that $L$ is a $\gl$-regular operator $L$ on $\mathsf T^2$ with  a single eigenvalue $\lambda$ of multiplicity 2.  The cases with constant and non-constant $\lambda$ are essentially different. If $\lambda=\mathrm{const}$, then w.l.o.g. we may assume that $\lambda = 0$, i.e., $L$ is nilpotent.

\begin{itemize}
\item Consider two vector fields $\xi$ and $\eta$ on $\mathsf{T}^2$ which are linearly independent at each point and define $L$ as follows:
$$
L(\xi)=0, \quad L(\eta)= \xi.
$$
Then $L$ is a $\gl$-regular nilpotent Nijenhuis operator on $\mathsf{T}^2$ (notice that any nilpotent operator in dimension 2 is automatically Nijenhuis). 

\item  The case with a non-constant eigenvalue on $\mathsf{T}^2$ can be modelled as follows:
$$
\begin{pmatrix}
f(\phi_2) & g(\phi_1,\phi_2) \\
0 & f(\phi_2)
\end{pmatrix},\qquad g(\phi_1,\phi_2)>0,
$$
where $\phi_1, \phi_2$ denote usual angle coordinates on the torus as above.
\end{itemize}
}\end{Ex}

Finally, we notice that examples corresponding to Case 4  of Theorem \ref{main:3} on the torus $\mathsf T^2$ and Klein bottle $\mathsf K^2 = \mathsf T^2/\sigma$ can be defined by the same formula as \eqref{eq:bols13}. The only difference is that now $f(\phi_2)$  vanish for some $\phi_2$ (but then $g(\phi_1,\phi_2)$ does not!).  The operator $L$ will become nilpotent at such points, which will be automatically singular from series $M$.  Notice that the topological structure of the eigenvector field $\xi$ related to the eigenvalue $f(\phi_2)$ may now be rather non-trivial in contrast to the case when $f\ne 0$. 

We conjecture that the above list of examples essentially exhausts all possible {\it real-analytic} Nijenhuis operators on closed two-dimensional surfaces.  In the smooth case, however, there are essentially different possibilities.

\section{Appendix:  On integration of some hydrodynamic type systems}\label{sect:appendix}

Theorem \ref{main:2} is closely related  to another fundamental problem of solving systems of quasilinear PDEs of the form 
\begin{equation}
\label{eq:bols21}
u^i_t = L^i_j (u) u^j_x,
\end{equation}
or a more general system 
\begin{equation}
\label{eq:bols22}
u_{x_{j-1}} = L (u) u_{x_j},    \quad  j = 1,\dots, m. 
\end{equation}
If $L(u)=\Bigl( L^i_j(u)\Bigr)$ is a Nijenhuis operator, such a system is known to be integrable.    If $L(x)$ is $\R$-diagonalisable, then after rewriting $L$ in diagonal form by Haantjes theorem,  the system \eqref{eq:bols21} splits into $n$ uncoupled Hopf equations  $u^i_t = \lambda(u^i)  u^i_x$  that can be easily solved (see e.g. \cite[Section 2.3]{chechkin}). However, at those points where $L$ is not diagonalisable,  this obvious idea does not work directly, although local analytic solutions still exist as  \eqref{eq:bols21} is a system of Cauchy-Kovalevskaya type. 

This Appendix provides a helpful tool to integrate \eqref{eq:bols21} and \eqref{eq:bols22} near singular points where a Nijenhuis operator is not diagonalisable.  It can be used most effectively if the operator $L$ is $\gl$-regular, but does not require this assumption.

System \eqref{eq:bols21} determines the dynamics of field variables $u^i$ which, in turn, describe the dynamics of $f_i(u)$, the coefficients of the characteristic polynomial of $L$: 
$$
\chi_L(\lambda) = \lambda^n - f_1(u)\lambda^{n-1} - \dots - f_{n-1}(u)\lambda - f_n(u).
$$
Unlike eigenvalues of $L$,  these functions are everywhere smooth and for this reason are more suitable for analysis of solutions of 
\eqref{eq:bols21} near those points where the eigenvalues collide.

Let  $L(u)$ be a Nijenhuis operator (not necessarily differentially non-degenerate or $\gl$-regular).  We want to solve  \eqref{eq:bols21}, i.e., find  the solution $u(t,x)$ with given initial conditions 
\begin{equation}
\label{init_cond_for_u}
u^i(0,x)=u^i_0(x).
\end{equation}

Instead of solving this problem we shall try to solve a ``simpler'' problem  (which is equivalent to it if $L$ is differentially non-degenerate).
Namely,  instead of $u(t,x)$ we will be looking for  $f_i(u(t,x))$, the coefficients of the characteristic polynomial, with the corresponding initial condition
\begin{equation}
\label{init_cond_for_f}
f_i(u(0,x))=f_i(u^i_0(x))= v_i(x).
\end{equation}

Notice that the dependence $f$ of $u$ is explicit, but the inversion is not always possible.  In this setting we have the following theorem.

\begin{Theorem}
\label{magic}
 For $n$ real analytic functions $v_1(t), \dots, v_n(t)$ defined from    \eqref{init_cond_for_f},
 consider the function 
 $$
 r(\lambda,\mu) = \lambda^n  -  v_1(\mu) \lambda^{n-1} - v_2(\mu) \lambda^{n-2} - \dots - v_{n-1}(\mu) \lambda - v_n(\mu)
 $$ 
 and the matrix relation 
 $$
 r(L, M) = 0,
 $$
 where $M = t L +  x\operatorname{Id}$ and $L$ is a $\gl$-regular  $n\times n$ matrix. 
 Then
 \begin{itemize}
\item[$(1)$] From this matrix relation,  the coefficients $f_1,\dots, f_n$ of the characteristic polynomial of $L$ can be uniquely expressed in a neighbourhood of $(t,x)=(0,0)$  as real analytic functions in $t, x$ (by Implicit Function Theorem).   

\item[$(2)$] The functions $f_1(t,x),\dots,f_n(t,x)$ so obtained are the coefficients of the characteristic polynomial of $L(u(t,x))$  where  
$u(t,x)$ is the solution of \eqref{eq:bols21} with initial condition \eqref{init_cond_for_u}.
 \end{itemize}
\end{Theorem}

\begin{proof}

We first derive the system of PDEs that governs the dynamics of $f_1,\dots, f_n$.
These equations are easy to describe. Indeed, every Nijenhuis operator $L(u)$ satisfies the relation (see \cite[Proposition 2.2]{Nijenhuis1})
$$
\left(  \frac{\partial f}{\partial u}\right) L(u) = \Lfirst(f) \left(  \frac{\partial f}{\partial u}\right). 
$$
Multiplying the both sides of this matrix relation with $u_x$ we get
$$
\left(  \frac{\partial f}{\partial u}\right) L(u) u_x = \Lfirst(f) \left(  \frac{\partial f}{\partial u}\right) u_x
 $$
and then using \eqref{eq:bols21} and the standard chain rule $\left(  \frac{\partial f}{\partial u}\right)u_x = f_x$,  
$\left(  \frac{\partial f}{\partial u}\right) u_t = f_t$:
\begin{equation}
\label{eq:bols23}
f_t = \Lfirst(f) f_x.
\end{equation}

If we denote $x=x^n$ and $t=x^{n-1}$ (index, not power!), then \eqref{eq:bols23} coincides with one the equations from \eqref{eq:20}.  Since we know the description for all (local) solutions of \eqref{eq:20},   we can simply obtain the required solution of \eqref{eq:bols23} by setting 
$x^1=\dots=x^{n-2}=0$ in the formulas given in Theorem \ref{main:2}.  This immediately leads to the conclusion of Theorem \ref{magic}.
\end{proof}

\begin{Remark}{\rm
Can we reconstruct the corresponding solution $u(x,t)$ of the original system \eqref{eq:bols21} from $f(x,t)$?    If $L$  is differentially non-degenerate, i.e., $f_1(u),\dots, f_n(u)$ are functionally independent as functions of $u=(u_1,\dots, u_n)$, then the answer is positive. This can be done just by inverting the map $u \mapsto f=f(u)$.  Thus, for differentially non-degenerate Nijenhuis operators,  Theorem \ref{magic} gives a way for solving  \eqref{eq:bols21} near those points where the eigenvalues of $L$ collide.  

If the differential non-degeneracy condition violates at some points (but not identically), then the relation $f (x,t) =  f(u(x,t))$ will still provides strong algrebraic restrictions  which might be sufficient for unique reconstruction of $u(x,t)$. 

Another useful application of Theorem \ref{magic} might be detection of singularities  (like gradient catastrophe) for solutions $u(x,t)$.  If we can find $f(x,t)$ and then observe that this solution is singular at a certain point $(x,t)$, then $u(x,t)$ will be singular too.
}\end{Remark}

\begin{Remark}{\rm
Theorem \ref{magic}  can be naturally adapted for system \eqref{eq:bols22}. We simply need to replace $(t,x)$ with $(x^1,\dots, x^m)$ and set $M=x^1 L^{m-1} + \dots + x^{m-1} L + x^m \operatorname{Id}$. The initial conditions should be taken in the form $u^i(0,\dots,0, x^m)= u^i_0(x^m)$. 
}\end{Remark}

\end{document}